%% file: momenMatch.tex
\documentclass[preprint,3p]{elsarticle}
\usepackage{stmaryrd}
\SetSymbolFont{stmry}{bold}{U}{stmry}{m}{n}
\usepackage{amsmath,bbm,mathbbol,mathtools}
\usepackage{amssymb,amsfonts,amsthm}
\usepackage{tabularx,lmodern}
\usepackage{color,soul,enumitem}
\usepackage{epsfig,epstopdf,color}
\usepackage{hyperref,graphicx}
\usepackage{algorithm,algorithmic,multirow}
\usepackage{flushend}
\usepackage{verbatim}
\usepackage{caption}
\usepackage{graphicx, subfig}
\usepackage{arydshln}
\usepackage[level]{datetime}
\usepackage{float}
\usepackage{booktabs}
\usepackage{tabularx}
\usepackage{natbib}
\usepackage{relsize}
\usepackage{ulem}
\usepackage{setspace}
\usepackage{epstopdf}
\usepackage{appendix}
\usepackage{bm}
\usepackage{adjustbox}
\usepackage{array}
\usepackage{empheq}

\DeclareMathOperator\supp{supp}

\include{command}

\begin{document}

\title{Fast convolution solvers using moment-matching
}

\author[eit]{Xin Liu}
\ead{xinliu@eitech.edu.cn}

\author[scu]{Qinglin Tang}
\ead{qinglin\_tang@scu.edu.cn}

\author[tju1,tju2]{Yong Zhang\corref{5}}
\ead{Zhang\_Yong@tju.edu.cn}

\address[eit]{Eastern Institute for Advanced Study, Eastern Institute of Technology, Ningbo,  China, 315200,}
\address[scu]{School of Mathematics, Sichuan University, ChengDu, China, 610064,}
\address[tju1]{Center for Applied Mathematics and KL-AAGDM, Tianjin University, Tianjin, China, 300072,}
\address[tju2]{State Key Laboratory of Synthetic Biology,   Tianjin University, Tianjin, China, 300072.}

\cortext[5]{Corresponding author.}

\begin{abstract}
We propose two easy-to-implement fast algorithms based on  moment-matching to compute the nonlocal potential
$\varphi(\bx)=(U\ast \rho)(\bx)$ on bounded domain,
where the kernel $U$ is singular at the origin and the density $\rho$ is a fast-decaying smooth function.
Each method requires merely minor modifications to commonly-used existing methods, i.e., the sine spectral/Fourier quadrature method,
and achieves a much better convergence rate.
The key lies in the introduction of a smooth auxiliary function $\rho_1$ whose moments match those of the density up to an integer order $m$.
Specifically, $\rho_1$ is constructed using Gaussian function in an explicit way
and the associated potential can be calculated analytically.
The moments of residual density vanish up to order $m$,
and the corresponding residual potential $U \ast (\rho-\rho_1)$ decays much faster
than the original potential $\varphi$ at the far field.
As for the residual potential evaluation,
for classical kernels (e.g., the Coulomb kernel),
we solve a differential/pseudo-differential equation on a rectangular domain with homogeneous Dirichlet boundary conditions via sine pseudospectral method,
and achieve an arbitrary high convergence rate.
While, for general kernels, the regularity of Fourier integrand increase by $m$ thanks to the moments-vanishing property,
therefore, the standard trapezoidal rule/midpoint quadrature also converges much faster.
To gain a better numerical performance, we utilize the domain expansion technique to obtain better accuracy,
and improve the efficiency by simplifying the quadrature into one discrete convolution and applying Fast Fourier Transform (FFT)
to a double-sized vector.
Rigorous error estimates and extensive numerical investigations showcase the accuracy and efficiency for different nonlocal potentials.

\end{abstract}


\maketitle


\section{Introduction}
 In this paper,
we focuses on the accurate and efficient evaluation of convolution-type nonlocal potentials
\be\label{convolution}
\varphi(\bx) = (U \ast \rho)(\bx) = \int_{\mathbb{R}^d} U(\bx-\by) \rho(\by) {\rm d} \by,\quad  \bx \in \mathbb{R}^d,
\ee
where $ d $ is the spatial dimension,
the density $\rho(\bx)$ is a fast-decaying smooth function and the kernel function $U(\bx)$ is usually singular at the origin.
 This is a vital
and prominent area of research in the science and engineering community.

The convolution \eqref{convolution} can be represented equivalently as a Fourier integral
\bea
\varphi(\bx)=\frac{1}{(2 \pi)^d} \int_{\mathbb{R}^d} \widehat{U}(\bk)  \widehat{\rho}(\bk) e^{i \bk \cdot \bx}{\rm d}\bk,
\label{FourierIntegral}
\eea
 where
$\widehat{f}(\bk)= \int_{\mathbb{R}^d} f(\bx) e^{-i \bk \cdot \bx} {\rm d} \bx$ is the Fourier transform of function $f(\bx)$.
The Fourier transform of density $\widehat{\rho}(\bk)$ is smooth and decays fast at the far field.
The Fourier transform of the kernel $\widehat{U}(\bk)$ is singular too,
and, sometimes, the singularity is too strong such that the above integral is not well-defined for general density.
A widely-used approach, named as \textit{Fourier spectral method} (FS) \cite{FourierMethod} hereafter,
discretizes the integral using trapezoidal rule and the resulting quadrature is implemented by Fast Fourier Transform (FFT).
Naive ignorance or improper treatment of the singularity shall render severe accuracy lost \cite{FourierMethod} or
reach a saturation, which is also named as ``numerical locking phenomenon'' in \cite{NUFFTCoulomb}.

As far as we know, several efficient and accurate algorithms have been proposed to deal with the singularity, such as
the NonUniform Fast Fourier Transform method (NUFFT)-based algorithm \cite{NUFFT}, Gaussian-Summation method (GauSum) \cite{GauSum, fastGausstran},
Kernel Truncation Method (KTM) \cite{FastConvGreengard, Liu2022optimal, CoulombCutoff},
Anisotropic Truncated Kernel Method (ATKM) \cite{atkm} and the newly-developed far-field smooth approximation (FSA)-based method \cite{FSA}.
All these methods focus on the integral form \eqref{convolution} or
\eqref{FourierIntegral}, and achieve spectral accuracy while maintaining FFT-like efficiency.

Besides the integral representation, one common way is to solve a partial differential or fractional differential equation, such as
the 3D Poisson and 2D Coulomb potential \cite{BaoBECRev,DSTMethod,DimRdct}, on bounded domain with appropriate boundary conditions.
For example, to solve the 3D Poisson equation with sine pseudospectral (SP) method \cite{DSTMethod}, we first
truncate the whole space into a bounded rectangular domain $\Omega := [-L,L]^d$ and impose homogeneous Dirichlet boundary conditions.
Then we apply sine pseudospectral method and utilize the discrete sine transform to help achieve great efficiency.
Unfortunately, the numerical accuracy cease to improve but become saturated instead as the mesh size tends smaller on a fixed computational domain.
The saturated accuracy is dominated by errors caused by boundary condition approximation,
which is determined by the asymptotics of the potential $\varphi$ at the far field,
and is only first-order with respect to the inverse of domain size in this case.

Despite the ``numerical locking'', the SP method gains great popularity in physics community due to its simplicity and implementation easiness.
The ``locked'' accuracy of the numerical solution will become smaller once the boundary condition approximation gets more accurate.
In this paper, we introduce an auxiliary function to neutralize the density by moment-matching technique.
The associated potential can be computed analytically and the residual potential decays much faster to zero at the far field.
Applying SP to solve PDE with homogeneous Dirichlet boundary conditions for the residual potential will alleviate such locking phenomenon
or even eliminate it up to a given high precision.

To be specific, if the density satisfies the following vanishing moment conditions
\begin{equation} \label{MomVan1}
  \int_{\mathbb{R}^d} \rho(\bx)\, \bx^{\bm{\alpha}} \rm{d} \bx =0, ~~ \left| \bm{\alpha}  \right|=0, 1, \dots, m, \quad m \in \mathbb{N},
\end{equation}
where $\bm{\alpha}=(\alpha_1, \dots, \alpha_d) \in \mathbb{N}^d$ is a multi-index,
$\bx^{\bm{\alpha}}= \prod\limits_{i=1}^{d} x_i^{\alpha_i} $ and $\left|\bm{\alpha}\right|=\sum\limits_{i=1}^{d} \alpha_i $,
using standard asymptotic analysis, we can prove that the potential decays as fast as
\bea
\label{FarfieldBeha}
|\varphi(\bx)| ~~~~ \sim ~~ \sum_{|\bm{\alpha}| = m+1} |\partial^{\bm{\alpha}} U(\bx)|, \qquad \quad  |\bx| \rightarrow \infty,
\eea where $\partial^{\bm{\alpha}} U = \partial_{x_1}^{\alpha_1} \cdots \partial_{x_d}^{\alpha_d}~ U$.
Take the 3D Poisson for example, the potential $\varphi$ decays as fast as $1/|\bx|^{m+2}$ at the far field, therefore,
the optimal accuracy achieved by SP method is polynomially small as $L^{-(m+2)}$.

\vspace{0.5em}

Most of the important potentials can be reformulated as PDE or fractional PDE
 \bea
\label{PDEVersion}
\mathcal{L} \varphi(\bx) = \rho(\bx), \quad \bx \in \mathbb{R}^d,
\eea
where $\mathcal{L}$ is a linear differential operator, for example,
$\mathcal{L} = -\Delta $ (2D/3D Poisson)
or $\sqrt{-\Delta}$ (2D Coulomb). A more detailed list on linear operators of convolutions is presented in Table \ref{tab:DiffAsyKernel}.
For this class of nonlocal potentials,
we introduce an auxiliary function $\rho_1(\bx)$ such that its associated potential $\varphi_1 := U \ast \rho_1$ can be computed analytically
and its moments match those of the density up to order $m$, i.e.,
  \begin{equation}
  \label{MM-condition}
  \int_{\mathbb{R}^d} \rho_1(\bx) \bx^{\bm{\alpha}} \rm{d} \bx =
  \int_{\mathbb{R}^d} \rho(\bx) \bx^{\bm{\alpha}} \rm{d} \bx
  , ~~ \left| \bm{\alpha}  \right|=0, 1, \dots, m.
\end{equation}
The residual density $\rho_2:= \rho-\rho_1$ automatically satisfies the vanishing  moment conditions
\eqref{MomVan1} and its resulting potential $\varphi_2: = U \ast \rho_2$ is readily solved by SP method.

The construction of auxiliary function turns into the core problem and is of essential importance.
So far as we know, there are not so many literature focusing on such topic \cite{ExpandDomain,NUFFTCoulomb}.
In \cite{ExpandDomain}, the author proposed compact polynomial mollifiers as the auxiliary function to compute the  Poisson potential in the whole space.
Such a polynomial-type function is non-smooth, so is the residual density.
On the one hand, the convergence rate of SP method, when applied to solve $\varphi_2$,
is limited by the regularity of $\rho_2$,
thus, it requires a small spatial mesh size $h$ to achieve high precision. On the other hand,
the discretization of Fourier integral should be carried out on a relatively large domain $[-\pi/h, \pi/h]^d$,
because the Fourier integrand $\widehat{U}(\bk) \widehat{\rho_2}(\bk) e^{i\bk\cdot \bx}$ decays only polynomially fast
at the far field.
Therefore, accurate evaluation of residual potential $\varphi_2$ requires a very small mesh size $h$
and shall eventually suffer from a severe efficiency degradation. It is more natural to adopt smooth auxiliary function
so to overcome such difficulty.
In fact, Bao et al. \cite{NUFFTCoulomb} utilized Gaussian and its first-order derivative to construct
a smooth auxiliary function successfully for the 2D Poisson potential.
Based on these facts, we propose one auxiliary function as a linear combination of Gaussian and its derivatives up to order $m$, i.e.,
\bea
\rho_1(\bx) =\sum\limits_{|\bm{\alpha}|=0}^m \gamma_{\bm{\alpha}} \frac{\partial^{\bm{\alpha}} G(\bx)}{\partial \bx^{\bm{\alpha}}}, \quad \text{with} \quad G(\bx)=\frac{1}{(2 \pi \sigma^2)^{d/2}} e^{-\frac{|\bx|^2}{2\sigma^2}}.
\label{AuxilFunc}
\eea
The coefficients $\gamma_{\bm{\alpha}}$ are determined by the moment-matching conditions  \eqref{MM-condition}
and are computed analytically by solving a linear system.
With the above Gaussian-type auxiliary function, the analytical computation of $\varphi_1$ is quite feasible for the common kernels,
including Poisson, Coulomb and Biharmonic kernels.
While for the general kernels, one may resort to high-order numerical integration, for example, the Gauss-Kronrod quadrature.
We skip details here for a more smooth presentation
and refer the readers to subsection \ref{subsec:AuxiFunc}.

Apart from the above common potentials,
there exists another large class of potentials that can not be written as a PDE or fractional PDE \eqref{PDEVersion} but
given by Fourier integrals \eqref{FourierIntegral}.
As stated before, with a simple computation, we know that the first $m$-th order Taylor expansion of $\widehat{\rho_2}(\bk)$ at $\bk=0 $ vanish, that is,
 \bea
 \label{TayExp}
    \frac{\partial^{\bm{\alpha}} \widehat{\rho_2}(\bk)}{\partial \bk^{\bm{\alpha}}} \bigg |_{\bk = 0} = \int_{\mathbb{R}^d}  \rho_2(\bx) \bx^{\bm{\alpha}} \rm{d} \bx = 0,
    \eea
 thus the regularity of Fourier integrand $\widehat{U}(\bk) \widehat{\rho_2}(\bk) e^{\im \bk \cdot \bx}$ gets elevated by order $m$.
 As a result, the accuracy of quadrature, such as the trapezoidal or midpoint rule, for the Fourier integral is  improved to
 $\mathcal{O}((\Delta k)^p)$, where $\Delta k$ is the mesh size
 in Fourier space and the integer $p$ depends on the moment-matching order $m$ and kernel $U$.
 Clearly, the numerical accuracy can be improved by increasing the convergence order $p$ or
 decreasing the mesh size $\Delta k  = \pi/L$, equivalently, increasing the domain size $L$.

In this article, we derive rigorous error analysis for both methods, and prove that the convergence order $p$ increases
as $m$ gets larger.
For example, for the 3D Poisson potential, the saturated accuracy achieved with SP method scales as $\mathcal{O}(L^{-1})$,
while the optimal accuracy with our proposed method is improved to $\mathcal{O}(L^{-(m+2)})$.
Most notably, when applied to the 2D Poisson potential, the SP method is \textit{not effective at all }
because the potential grows logarithmically fast \cite{ArtifiBoundCond},
while, the optimal accuracy achieved with our method is as small as $\mathcal{O}(L^{-(m+1)})$.

For a given integer $m$, to further improve the accuracy, we adopt the domain expansion \cite{ExpandDomain} and denote the so-called \textit{domain expansion factor} by $S$.
The memory costs and the computational complexity are $\mathcal{O}((SN)^d)$ and $\mathcal{O}((SN)^d \log (SN)^d)$ respectively, where
$N$ is the number of grids in each spatial dimension.  In a companion paper \cite{RecursiveExpandDomain},
Anderson  proposed a recursive expanding domain method to
help reduce the memory costs down to $\mathcal{O}(S^{d-1} N^{d})$ and the complexity, which becomes more favorable with $S$ as
$$\mathcal{O}((S^{d-1} N^d) \log (S^{d-1} N^d)).$$
Clearly, a large $S$ helps improve the accuracy at the sacrifice of efficiency.

Actually, similar to ATKM \cite{atkm} and FSA \cite{FSA}, both methods proposed here can be
simplified to a discrete convolution with tensor explicitly given by  Eqn. \eqref{TensorMidFou} or Eqn. \eqref{TensorHomDirSin}.
Once the tensor is computed out \textit{once for all} in the precomputation step, the essential execution involves merely a pair of FFT/iFFT of a double-sized density.
The computational complexity becomes independent of  $S$ and is greatly reduced to
$$\mathcal{O}((2N)^d \log (2N)^d).$$ Such tensor acceleration technique greatly improves the efficiency, especially when the potentials are frequently evaluated with the same numerical setup, such as in computing the ground state and dynamics of the nonlocal Schr\"{o}dinger equation \cite{BaoBECRev, TrappedAtomic_dipole}.

\vspace{0.5em}

The paper is organized as follows. In Section \ref{sec:FastSolvers},
 we present two simple and high-order fast solvers for nonlocal potential evaluation,
 along with tensor-based acceleration.   Error estimates are provided in Section  \ref{sec:ErrEst}.
  In Section \ref{sec:NumRes},
  extensive numerical investigations are
conducted  to demonstrate the accuracy and efficiency for various kernels. Finally, some conclusions are drawn in Section \ref{Conclusion}.

\section{Fast convolution  solvers} \label{sec:FastSolvers}
In this section, we develop fast convolution solvers
using moment-matching technique for the density.
The key observation is that the vanishing of the moments of $\rho$
up to order $m$ provides the following advantages:
 \begin{equation}
 \label{AdvanMomVan}
\begin{aligned}
&(a)~\mbox{The potential generated by $\rho$ decays faster at the far field}. \\[0.3em]
&(b)~\mbox{The regularity of integrand in Fourier integral}~ \eqref{FourierIntegral} ~ \mbox{gets elevated}.
\end{aligned}
\end{equation}

Advantage (a) follows from standard asymptotic analysis of the potential at the far-field, with details provided in Lemma~\ref{lem:farAsyAna} below.
 This implies that the homogeneous Dirichlet boundary condition approximation becomes more accurate, and it is then very convenient to apply sine pseudospectral method.
 Advantage (b) follows from Eqn. \eqref{TayExp}, which ensures that the discretization error in evaluating \eqref{FourierIntegral} decays more rapidly with respect to the mesh size in Fourier space.

Based on the above facts, we introduce an auxiliary function $\rho_1$, whose potential can be
computed with ease, such that the moments of  $\rho_2 = \rho-\rho_1$ vanish up to order $m$.
For clarity, we provide a simple procedure for solving $\varphi = U \ast \rho$ consisting of the following steps:

\floatname{algorithm}{Procedure}
\begin{algorithm}
\renewcommand{\thealgorithm}{}
\caption{ Moment-matching method for potentials $\varphi(\bx)= (U \ast \rho)(\bx)$. }
\label{ProModDST}
\begin{algorithmic}[1]
\STATE Create a auxiliary function $\rho_1(\bx) $ so that the moments of $\rho_2(\bx)=\rho(\bx)-\rho_1(\bx)$ vanish up to order $m$:
\bea
\int_{\mathbb{R}^d} \rho_2(\bx) \bx^{\bm{\alpha}} \rm{d} \bx =0, ~~ \left| \bm{\alpha}  \right|=0, 1, \dots, m.
\label{moment}
\eea
\\[0.3em]
\STATE  Compute $\varphi_1(\bx) = (U \ast \rho_1)(\bx)$ analytically. \\[0.3em]
\STATE Compute  $\varphi_2(\bx) = (U \ast \rho_2)(\bx)$
 using the SP method or FS method.\\[0.3em]
\STATE  Set $\varphi = \varphi_1 + \varphi_2$ .
\end{algorithmic}
\end{algorithm}

\vspace{0.5em}

\begin{lem}[\textbf{Far-field asymptotic analysis}]
\label{lem:farAsyAna}
For a smooth and  numerical compactly supported function $\rho(\bx)$, the asymptotic behavior of the nonlocal potential at the far field reads as follows
 \bea
\label{Asymptotic}
| \varphi(\bx)| ~~~~\sim ~~ \sum_{|\bm{\alpha}| = m+1}|\partial^{\bm{\alpha}} U(\bx)|, \qquad \mbox{ as }\quad |\bx| \rightarrow \infty,
\eea
if the moments of $\rho(\bx)$ vanish up to order $m$, i.e.,
\begin{equation} \label{MomVan}
  \int_{\mathbb{R}^d} \rho(\bx) \bx^{\bm{\alpha}} \rm{d} \bx =0, ~~ \left| \bm{\alpha}  \right|=0, 1, \dots, m.
\end{equation}
\end{lem}

\begin{proof}
Without loss of generality, we assume that $\supp \{\rho\} \subsetneq  \Omega_0$, and define \(\Omega\) as an expanded domain, i.e., $\Omega_0 \subsetneq  \Omega$.
Using  the Taylor formula with integral remainder for kernel $U$, we have
\bea
\label{KernelTaylor}
U(\bx-\by)= \sum_{k=0}^{m} \sum_{|\bm{\alpha}|=k} \frac{1}{\bm{\alpha} !}(-1)^{|\bm{\alpha}|} \partial^{\bm{\alpha}} U(\bx)\by^{\bm{\alpha}} + r_m, \quad \bx \in \partial\Omega,~~  \by \in \Omega_0,
\eea
where $\bm{\alpha}=(\alpha_1,  \alpha_2, \dots, \alpha_d) \in \mathbb{N}^d$ is a multi-index, $\left|\bm{\alpha}\right|=\sum\limits_{i=1}^{d} \alpha_i $, $\by^{\bm{\alpha}}= \prod\limits_{i=1}^{d} y_i^{\alpha_i} $,
$\partial^{\bm{\alpha}} U = \partial_{x_1}^{\alpha_1} \cdots \partial_{x_d}^{\alpha_d}~U$, $\bm{\alpha} ! =\alpha_1 ! \cdots \alpha_d !$  and the integral remainder
 \beas
 r_m=(m+1) \sum_{|\bm{\beta}| = m+1}\frac{1}{\bm{\beta} !} (-1)^{m+1} \by^{\bm{\beta}} \int_0^1 \partial^{\bm{\beta}} U(\bx-s \by) (1-s)^m \rm{d}s.
 \eeas
Using Eqn. \eqref{KernelTaylor} and $\supp \{\rho\} \subsetneq  \Omega_0$, for $\bx \in \partial \Omega$, we obtain
\beas
\varphi(\bx)&=& \int_{\mathbb{R}^d} U(\bx-\by) \rho(\by) \rm{d} \by = \int_{\Omega_0} U(\bx-\by) \rho(\by) \rm{d} \by \\
 &=& \sum_{k=0}^{m} \sum_{|\bm{\alpha}|=k} \frac{1}{\bm{\alpha} !}(-1)^{|\bm{\alpha}|} \partial^{\bm{\alpha}} U(\bx) \left[ \int_{\mathbb{R}^d} \rho(\by) \by^{\bm{\alpha}}  {\rm{d} \by} \right] +\int_{\Omega_0} r_m \rho(\by) \rm{d} \by.
\eeas
Using the moment conditions \eqref{MomVan}, we have
\bea
| \varphi(\bx)|&=& \left| \int_{\Omega_0} r_m \rho(\by) \rm{d} \by \right|
 \lesssim \sum_{|\bm{\beta}| = m+1} \max_{\substack{0 \le s \le 1}} \left| \int_{\Omega_0}  \partial^{\bm{\beta}} U(\bx-s \by)\by^{\bm{\beta}} \rho(\by)  \rm{d} \by  \right| \notag\\
 &\lesssim& \sum_{|\bm{\beta}| = m+1} \max_{\substack{ \by \in \Omega_0 \\ 0 \le s \le 1}}  \left| {\partial^{\bm{\beta}} U(\bx-s \by)  }\right|, ~~ \bx \in \partial \Omega,  \label{FarAsymMax}\\
 &\sim & \sum_{|\bm{\beta}| = m+1} |\partial^{\bm{\beta}} U(\bx)|, \qquad |\bx| \rightarrow \infty. \notag
\eea
\end{proof}

 \subsection{Auxiliary function}
 \label{subsec:AuxiFunc}
To ensure the computational viability of
 the proposed procedure, we present a specific scheme for constructing the auxiliary function $\rho_1(\bx)$ and its induced potential $\varphi_1(\bx)$.
 There are various ways to construct such a function, such as compact polynomial mollifiers \cite{ExpandDomain}. However, these mollifiers are not smooth and shall eventually lead to severe efficiency degradation. To overcome this issue, we construct the auxiliary function as a linear combination of smooth Gaussian and its derivatives up to order $m$, as given in \eqref{AuxilFunc}.
The coefficients $\gamma_{\bm{\alpha}}$  are obtained by solving the following linear system
\bea
\sum_{|\bm{\alpha}|=0}^{m} \left[ \int_{\mathbb{R}^d} \frac{\partial^{\bm{\alpha}} G(\bx)}{\partial \bx^{\bm{\alpha}}} \bx^{\bm{\beta}} \rm{d} \bx \right] \gamma_{\bm{\alpha}}= \int_{\mathbb{R}^d} \rho(\bx) \bx^{\bm{\beta}} \rm{d} \bx ~~~~ \text{for}~ |\bm{\beta}|=0,\dots,m.
\label{LineEqu}
\eea
In fact, the coefficient matrix $H_{\bm{\beta}}^{\bm{\alpha}}: = \int_{\mathbb{R}^d} \frac{\partial^{\bm{\alpha}} G(\bx)}{\partial \bx^{\bm{\alpha}}} \bx^{\bm{\beta}} \rm{d} \bx $ is the lower triangle, therefore the solution to equations \eqref{LineEqu} can be obtained analytically via back substitution. A detailed derivation is given in  \ref{coefficient}.

With the coefficients $\gamma_{\bm{\alpha}}$ determined, we can evaluate $ \varphi_1 = U \ast \rho_1 $
as follows
\bea
	\label{FirstExact}
	\varphi_1 = U \ast \rho_1 = U \ast \Big[\sum_{|\bm{\alpha}|=0}^m \gamma_{\bm{\alpha}} ~ \partial^{\bm{\alpha}} G\Big] =  \sum_{|\bm{\alpha}|=0}^m \gamma_{\bm{\alpha}} ~ \partial^{\bm{\alpha}}(U \ast G).
	\eea
An explicit analytical expression for $W: = U \ast G$ can be derived for several classes of potentials, including the 2D/3D Poisson, 2D Coulomb, 3D dipole-dipole interaction (DDI) and 2D/3D biharmonic potentials. Details are provided in ~\ref{anlyticalExpr}.
For general kernels, one may resort to numerical integration to obtain the accurate evaluation of $W$.
Here, we present a numerical scheme to compute the convolution of two radially symmetric functions in 2D, i.e., $U(\bx) = U(r)$ and $G(\bx)  = G(r)$ with $r = |\bx|$, then
\beas
W(\bx) = \int_0^{\infty} U(s) s \rm{d}s \int_0^{2 \pi} G \left(\sqrt{r^2+s^2-2 s r \cos \theta }\right) \rm{d} \theta,
 \eeas
 which can be calculated by adaptive Gauss-Kronrod quadrature.

\vspace{0.5em}

In the following, we present two simple numerical methods to compute $\varphi_2$ based on the properties \eqref{AdvanMomVan} induced by moment-vanishing. The first method employs the SP method to solve the corresponding differential equation subjected to homogeneous Dirichlet boundary conditions on a bound domain.  We refer to this as the $\mathbf{SP\text{-}MM}$ method, with details provided in Subsection \ref{sec:HomDirSin}.
  The second method discretizes the Fourier integral \eqref{FourierIntegral} using midpoint rule, and the quadrature is efficiently implemented via FFT. This is referred to as the $\mathbf{FS\text{-}MM}$ method, with details provided in Subsection \ref{sec:MidFou}.
  Both methods can achieve arbitrary accuracy
  with high efficiency thanks to discrete sine transform (DST) or FFT, and rigorous error estimates are provided in Section \ref{sec:ErrEst}.

 \subsection{SP-MM method} \label{sec:HomDirSin}
In this subsecction, we introduce the SP method
\cite{Shen2011Spectral} to discretize the following linear PDE
\bea
\left\{\ba{ll}
 \mathcal{L} \varphi_2(\bx)  = \rho_2(\bx), \quad \bx \in \Omega, \\[0.8em]
\varphi_2(\bx)\big|_{\bx \in \partial \Omega}=0,
\ea\right.
\label{odeHomDirRho2}
\eea
where the operator $\mathcal{L}$ contains only even-order derivatives.
Without loss of generality, we denote $\Omega = [-L,L]^d:= \bR_{L} $ and discrete it by $N \in 2 \mathbb{Z}^{+}$ equally spaced points in  each direction.  The mesh size is then given by $h = 2L/N $.
For simplicity, we only present the detailed scheme for 1D case. Extension to 2D/3D case is quite straightforward. The spatial grid points are given by $x_j = -L + j h$, $j =1, 2, \dots, N-1$.
Denote the space
\beas
X &=& \{f(x) \in C^{\infty}(\Omega),  ~ f(x) \big|_{x \in \partial \Omega} = 0  \}, \\
Y_{N} &=& \text{span} \{ \sin(\mu_p (x+L)), \quad p =1, 2, \dots, N-1  \},
\eeas
with $\mu_p = \pi p /(2 L)$.
 Let $P_{N}: X ~ \rightarrow Y_{N} $ be the standard project operator
as
\beas
(P_{N} f)(x) =  \sum_{p=1}^{N-1} \hat{f}^{s}_{p} \sin(\mu_p(x+L)) , \quad x \in \Omega,
\eeas
where $\hat{f}^{s}_{p}$, the sine transform coefficients, are defined as
\beas
\hat{f}^{s}_{p} =\frac{1}{L} \int_{-L}^{L} f(x)  \sin(\mu_p(x+L))  {\rm d} x.
\eeas
The above integral is well approximated by applying the trapezoidal rule, which is given explicitly as follows
\bea
\tilde{f}_{p}^{s} = \frac{2}{N} \sum_{j=1}^{N-1}  f(x_j) \sin \left(\frac{\pi j p}{N}\right).
\label{DisSin}
\eea
 Then
 the sine pseudospectral discretization for \eqref{odeHomDirRho2} reads as
\bea
\varphi_2^{N}(x_j) =\sum_{p=1}^{N-1}  \frac{(\tilde{\rho}_2)_{p}^{s}}{\widehat{\mathcal{L}}(\mu_p)} \sin \left(\frac{\pi j p}{N}\right),
\label{SineSpec}
\eea
where $(\tilde{\rho}_2)_{p}^{s}$,  the discrete sine transform coefficients, are defined as \eqref{DisSin}, and
$\widehat{\mathcal{L}}(\mu_p)$ represent the multiplication operator in Fourier space corresponding to the differential operator $\mathcal{L}$ in physical space, which is given by
\begin{equation*}
   \widehat{\mathcal{L}}(\mu_p) =  \langle \mathcal{L} e_p(x), e_p(x) \rangle / \langle  e_p(x), e_p(x) \rangle,
\end{equation*}
where $e_p(x) := \sin(\mu_p(x+L))$. For example, $\widehat{\mathcal{L}}(\mu_p) = (\mu_p)^2$ for the differential operator
$\mathcal{L} = - \partial_{xx}$.

Following the same procedure, we present a specific scheme for the two-dimensional problem. Define the physical index and grid points sets as
 \beas
 &&\mathcal{I}_{N} := \{ (j, k) \in \mathbb{N}^2~ |~ 1\le j \le N-1, ~ 1 \le k \le N-1 \}, \\[0.3em]
 &&\mathcal{T}_{\bx} := \{(x_j, y_k) := (-L+j h, -L+k h), \quad (j, k) \in \mathcal{I}_{N}  \}.
 \eeas
 The corresponding numerical scheme reads as follows
 \bea
\varphi_2^{N}(x_j, y_k) =\sum_{(p, q) \in \mathcal{I}_{N}}
\frac{(\tilde{\rho}_2)_{pq}^s}{\widehat{\mathcal{L}}(\mu_p, \mu_q )} \sin \left( \frac{\pi j p}{N} \right) \sin \left(  \frac{\pi k q}{N} \right),
\label{SineSpec2D}
\eea
 where
 \bea
(\tilde{\rho}_2)_{pq}^s = \frac{4}{N^2} \sum_{(j, k) \in \mathcal{I}_{N}} \rho_2(x_j, y_k)  \sin \left( \frac{\pi j p}{N} \right) \sin \left(  \frac{\pi k q}{N} \right).
\label{DisSin2D}
\eea

The proposed scheme is easy to implement and
can be efficiently executed using DST with $\mathcal{O} (N^d\log(N^d))$ computational complexity.
 It is applicable to a large class of kernels, and
we summarize some common kernels along with their corresponding differential operator $\mathcal{L}$  in Table \ref{tab:DiffAsyKernel}.

\begin{table}[htpb!]
\centering
\caption{The corresponding differential operator and the far-field asymptotics of $\varphi_2$ for common kernels.}
\label{tab:DiffAsyKernel}
\begin{adjustbox}{width=1.00\textwidth}
\begin{tabular}{llcc}
\toprule
  & $U(r)$ &  $\mathcal{L}$     &\qquad  \qquad  Far-field asymptotics of $\varphi_2$\\
\midrule
\rule{0pt}{16pt}
\footnotesize{\textbf{Coulomb}} &  {\footnotesize{$2D$:}} $ \frac{1}{2 \pi r} $ &  $\sqrt{-\Delta}$ & $  r^{-(m+2)}$ \\[6pt]
\midrule
\rule{0pt}{16pt}
\multirow{2}{*}{\footnotesize{\textbf{Poisson}}}
 &  {\footnotesize{$2D$:}} $ \frac{-1}{2 \pi} \ln(r) $ & $ -\Delta $ & $ r^{-(m+1)} $ \\[6pt]
 &  {\footnotesize{$3D$:}} $\frac{1}{4 \pi r} $ & $ -\Delta $ & $ r^{-(m+2)}$ \\[6pt]
 \midrule
 \rule{0pt}{16pt}
 \multirow{2}{*}{\footnotesize{\textbf{Biharmonic}}}
 & {\footnotesize{$2D$:}} $ \frac{- r^2 }{8 \pi  } \big[\ln(r)-1\big] $
 &  $-\Delta^2$ & $ r^{-(m-1)} $ \\[6pt]
  & {\footnotesize{$3D$:}} $ \frac{r}{8 \pi}$
 & $-\Delta^2 $ & $ r^{-m} $ \\[6pt]
 \midrule
 \rule{0pt}{16pt}
 \multirow{2}{*}{\footnotesize{\textbf{Yukawa}}}
 & {\footnotesize{$2D$:}} $ \frac{1}{2\pi  } K_0(\lambda r)$
 &  $-\Delta+ \lambda^2$ & $ e^{-r}~ r^{-1/2} $ \\[6pt]
  & {\footnotesize{$3D$:}} $ \frac{1}{4 \pi r} e^{-\lambda r}$
 & $-\Delta+ \lambda^2 $ & $e^{-r}~ r^{-1} $ \\[6pt]
\bottomrule
\end{tabular}
 \end{adjustbox}
\end{table}


%
%

\subsection{FS-MM method}\label{sec:MidFou}
In this subsection, we introduce FS-MM method to calculate
\bea
\label{Phi2FouInt}
\varphi_2(\bx)=\int_{\mathbb{R}^{d}} U(\bx-\by) \rho_2(\by) {\rm d} \by=\frac{1}{(2 \pi)^d} \int_{\mathbb{R}^d} \widehat{U}(\bk)~ \widehat{\rho_2}(\bk) ~e^{i \bk \cdot \bx}~ {\rm d} \bk.
\eea
Since the modified density $\rho_2$ is smooth and decays exponentially fast,
its Fourier transform $\widehat{\rho_2}(\bk)$ is also smooth and decays fast \cite{Shen2011Spectral}.
Therefore, it is reasonable to truncate the whole space $\mathbb R^d$ into
a large enough bounded domain $\mathcal{D}$.
The main idea of typical Fourier method is
 to discrete \eqref{Phi2FouInt} using the trapezoidal rule with a dual mesh $\Delta k=\pi/L$ on the truncation domain $\mathcal{D}=[-\pi/h, \pi/h]^d$.
 Usually, the Fourier transform of kernel $\widehat U(\bk)$ is singular at the origin. To avoid the origin, we adopt the
 midpoint rule instead of trapezoidal rule, and the resulting quadrature still can be efficiently accelerated using FFT within ${\mathcal O}( N^d \log(N^d) )$ complexity. Specifically, takeing 2D case as an example, we define the physical index and grid points sets as follows
 \beas
 &&\mathcal{J}_{N} := \{ (j, k) \in \mathbb{N}^2~ |~ -N/2\le j \le N/2-1, ~ -N/2 \le k \le N/2-1 \}, \\[0.3em]
 &&\mathcal{G}_{\bx} := \{ (x_j, y_k) := (j h, k h), \quad (j, k) \in \mathcal{J}_{N}  \}.
 \eeas
By discretizing \eqref{Phi2FouInt} using the midpoint rule with a dual mesh $\Delta k=\pi/L$ on the truncation domain $\mathcal{D}=[-\pi/h, \pi/h]^d$, we obtain
\bea
\qquad
\label{SchemeMid}
\varphi_2^N(\bx_{jk}) =
\frac{1}{(2 L) ^2}\hspace{-0.5em}
\sum_{(p, q) \in \mathcal{J}_{N}} \hspace{-0.6em}
\widehat{U}(\nu_{p+\frac{1}{2}}, \nu_{q+\frac{1}{2}}) \widehat{\rho_2}(\nu_{p+\frac{1}{2}}, \nu_{q+\frac{1}{2}}) e^{\frac{2\pi \im}{N}\big[(p+\frac{1}{2}) j + (q+\frac{1}{2}) k\big]}, \hspace{-0.9em}
\eea
where $\bx_{jk}:= (x_j, y_k)$, $\nu_p = (2 \pi p)/(2 L)$ and  $\nu_q = (2 \pi q)/(2 L)$. The Fourier transform   $\widehat{\rho_2}(\nu_p, \nu_q)$  is well approximated by applying the trapezoidal rule, and the resulting
summation is given as follows
\bea
\widetilde{\rho_2}(\nu_p, \nu_q)  =  h^2 \sum_{(j, k)\in \mathcal{J}_{N}}\rho(x_{j}, y_{k})~e^{-\frac{2\pi i}{N}(p j + q k)}.
\label{rhohatS}
\eea

\subsection{Domain expansion and tensor structure simplification}

Both the SP-MM and FS-MM methods are simple, efficient and can achieve the arbitrary high order with respect to
$L^{-1}$ (see Theorem \ref{thm:ErrHomDirSin}-\ref{thm:ErrMidFou}).
Therefore, the accuracy can be furthermore improved by expanding the computational domain. Without loss of generality, let the expanded domain be $\bR_{SL}:=  [-SL, SL]^d$, where  $S$ is referred to as the \textbf{domain expansion factor}.
Taking FS-MM method as an example,
 the scheme \eqref{SchemeMid} becomes
\bea
\qquad \quad
\label{SchemeMidExpand}
\varphi_2^N(\bx_{jk}) =
\frac{1}{(2 S L) ^2}\hspace{-0.8em}
\sum_{(p, q) \in \mathcal{J}_{S N}}\hspace{-0.8em}
\widehat{U}(\nu_{p+\frac{1}{2}}, \nu_{q+\frac{1}{2}}) \widetilde{\rho_2}(\nu_{p+\frac{1}{2}}, \nu_{q+\frac{1}{2}})e^{\frac{2\pi \im}{S N}\big[(p+\frac{1}{2}) j + (q+\frac{1}{2}) k\big]}.\hspace{-0.8em}
\eea

It is obvious that the accuracy can be improved by choosing a large $S$, but this bottlenecks the efficiency
 and puts a heavy burden on memory requirement.
Specifically, the memory costs and the computational complexity are $\mathcal{O}((SN)^d)$ and $\mathcal{O}((SN)^d \log (SN)^d)$ respectively.  In a companion paper \cite{RecursiveExpandDomain},
Anderson  proposed a recursive expanding domain method to
help reduce the memory costs to $\mathcal{O}(S^{d-1} N^{d})$ and the complexity to $\mathcal{O}((S^{d-1} N^d) \log (S^{d-1} N^d))$, which becomes more favorable with $S$.
Thus, while a larger $S$ improves accuracy, it inevitably imposes a trade-off in terms of efficiency.

 Fortunately, the resulting summation \eqref{SchemeMidExpand} can be equivalently reformulated as a discrete
convolution structure by plugging \eqref{rhohatS} into \eqref{SchemeMidExpand} and switching the summation order, that is,
\bea
\varphi_{2}^N(x_j, y_k)  =\sum_{(j', k') \in \mathcal{J}_N} T
 _{j-j', k-k'} ~(\rho_2)_{j', k'}.  \label{SchemeMidFouTen}
\eea
The tensor $T$ reads explicitly as
\bea
T_{j,k}= \frac{1}{ (S N)^2} \sum_{(p, q) \in \mathcal{J}_{S N}}
\widehat{U}(\nu_{p+\frac{1}{2}},\nu_{q+\frac{1}{2}}) ~ e^{\frac{2\pi \im}{S N} \big[(p+\frac{1}{2}) j+(q+\frac{1}{2}) k\big]},
\label{TensorMidFou}
\eea
and it can be computed out \textit{once for all} within $\mathcal{O}((SN)^2 \log((SN)^2))$ float operations using iFFT.

Once the tensor is available in the precomputation step, the effective computation is merely a pair of FFT and iFFT on a \textbf{double-sized} vector \cite{Liu2022optimal}.
The computational complexity becomes independent of  $S$ and is greatly reduced to
$$\mathcal{O}((2N)^d \log (2N)^d).$$  Such tensor acceleration technique greatly improves the efficiency,
particularly in scenarios where the potential is frequently evaluated under the same numerical setup, such as in computing the ground state and dynamics of the nonlocal Schr\"{o}dinger equation \cite{BaoBECRev, PfauDipolar}.
We refer to it as the $\mathbf{FS\text{-}Tensor}$ method.

\vspace{0.5em}

 For SP-MM method, using the identity $\sin(x) = (e^{\im x} - e^{-\im x})/(2\im)$ and the even-odd decomposition
 \eqref{oddEvenDeco} of density, as described in the proof of Theorem
 \ref{thm:ErrHomDirSin2DCoulomb},  it can equivalently be rewritten as $2^d$ discrete convolution summations. To be
 specific, for two-dimensional problem, we obtain
 \beas
 \varphi_{2}^N(x_j, y_k) ~ & = & \sum_{(j', k') \in \mathcal{J}_N} \Big[  T^{ee}
 _{j-j', k-k'} ~(\rho_2^{ee})_{j', k'} +  T^{eo}
 _{j-j', k-k'} ~(\rho_2^{eo})_{j', k'} + \\
&& \qquad \qquad ~~ T^{oe}
 _{j-j', k-k'} ~(\rho_2^{oe})_{j', k'} +  T^{oo}
 _{j-j', k-k'} ~(\rho_2^{oo})_{j', k'} \Big] \\[1em]
 &:=& T^{ee} \ast \rho_2^{ee} + T^{eo} \ast \rho_2^{eo}+T^{oe} \ast \rho_2^{oe}+ T^{oo} \ast \rho_2^{oo},
 \eeas
 where the tensors are given by
  \bea
T^{ee}_{j,k}&= &\frac{1}{ (S N)^2} \sum_{(p, q) \in \mathcal{J}_{S N}}
\Big(\widehat{\mathcal{L}}(\mu_{2p+1},\mu_{2q+1})\Big)^{-1}  e^{\frac{2\pi \im}{S N} \big[(p+\frac{1}{2}) j+(q+\frac{1}{2}) k\big]}, \label{TensorHomDirSin} \\[0.4em]
T^{eo}_{j,k}&= &\frac{1}{ (S N)^2} \sum_{\substack{(p, q) \in \mathcal{J}_{SN} \\ q \ne 0}}
\Big(\widehat{\mathcal{L}}(\mu_{2p+1},\mu_{2q})\Big)^{-1}  e^{\frac{2\pi \im}{S N} \big[(p+\frac{1}{2}) j+q k\big]},
\eea
\bea
T^{oe}_{j,k}&= &\frac{1}{ (S N)^2} \sum_{\substack{(p, q) \in \mathcal{J}_{SN} \\ p \ne 0}}
\Big(\widehat{\mathcal{L}}(\mu_{2p},\mu_{2q+1})\Big)^{-1}  e^{\frac{2\pi \im}{S N} \big[p j+(q+\frac{1}{2}) k\big]}, \\
T^{oo}_{j,k}&= &\frac{1}{ (S N)^2} \sum_{(p, q) \in \mathcal{J}_{SN} \setminus \{(0, 0)\}}
\Big(\widehat{\mathcal{L}}(\mu_{2p},\mu_{2q})\Big)^{-1}  e^{\frac{2\pi \im}{S N} \big[p j+q k\big]}.
\eea

Based on the relation between midpoint rule and rectangle rule given in Theorem~\ref{thm:ErrMidFou}, the following approximation is valid
by neglecting error terms of the same order and above, that is,
  \beas
   \varphi_{2}^N  &=& T^{ee} \ast \rho_2 + [T^{eo}-T^{ee}] \ast \rho_2^{eo}+
   [T^{oe}-T^{ee}] \ast \rho_2^{oe}+ [T^{oo}-T^{ee}] \ast \rho_2^{oo} \\[0.5em]
   &\approx& T^{ee} \ast \rho_2.
  \eeas
  Therefore, we can simplify the SP-MM method to a single discrete convolution summation \eqref{SchemeMidFouTen}
  with the tensor $T_{j, k} = T^{ee}_{j, k}$.
We refer to it as the $\mathbf{SP\text{-}Tensor}$ method.
For clarity,  we present a detailed step-by-step algorithm proposed in Algorithm \ref{AlgMidFouTen} using the 2D case.
\floatname{algorithm}{Algorithm}

\begin{algorithm}
\renewcommand{\thealgorithm}{1}
\caption{FS-Tensor/SP-Tensor method for $\varphi(\bx) = (U \ast \rho)(\bx)$. }
\label{AlgMidFouTen}
\begin{algorithmic}[1]
  \REQUIRE (i). Precompute $W(\bx):= [U \ast G](\bx)$ with $G(\bx)=\frac{1}{(2 \pi \sigma^2)^{d/2}} e^{-\frac{|\bx|^2}{2\sigma^2}}$.  \\[0.5em]
    \qquad ~~ (ii). Choose a $m \in \mathbb{N}$ and  precompute $\frac{\partial^{\bm{\alpha}} G(\bx)}{\partial \bx^{\bm{\alpha}}}$, $\frac{\partial^{\bm{\alpha}} W(\bx)}{\partial \bx^{\bm{\alpha}}}$ with $|\bm{\alpha}|=0, \dots, m$.  \\[0.5em]
    \quad ~~~~ (iii). Choose a $S \in \mathbb{N}$ and precompute the tensor $T$ by Eqn.\eqref{TensorMidFou} (FS) or \\ \qquad \quad ~~~~~ Eqn.\eqref{TensorHomDirSin} (SP)  via iFFT.  \\[0.5em]
\STATE Compute the coefficients $\gamma_{\bm{\alpha}}$ with $|\bm{\alpha}|=0, \dots, m$ (see~ \ref{coefficient}). \\[0.5em]
\STATE Compute $\rho_1(\bx) = \sum\limits_{|\bm{\alpha}|=0}^m \gamma_{\bm{\alpha}} \frac{\partial^{\bm{\alpha}} G(\bx)}{\partial \bx^{\bm{\alpha}}}   $  and $ \varphi_1 (\bx) =\sum\limits_{|\bm{\alpha}|=0}^m \gamma_{\bm{\alpha}} \frac{\partial^{\bm{\alpha}} W(\bx)}{\partial \bx^{\bm{\alpha}}}  $ exactly.\\[0.3em]
\STATE Set $\rho_2 = \rho - \rho_1$ and compute $ \varphi_2 $  by Eqn.\eqref{SchemeMidFouTen} via FFT/iFFT. \\[1.0em]
\STATE Compute $\varphi = \varphi_1+ \varphi_2$.
\end{algorithmic}
\end{algorithm}

\section{Error estimates} \label{sec:ErrEst}
In this section, we provide rigorous error estimates.
For the SP-MM method, the SP method achieves spectral accuracy for solving linear differential equations subject to homogeneous Dirichlet boundary conditions.
The boundary condition approximation error will dominate as the numerical discretization gets finer
because the nonlocal potential does not satisfy homogeneous Dirichlet boundary condition exactly but with an approximation error of some polynomial order in terms of domain size.
Detailed error analysis are provided in Theorem \ref{thm:ErrHomDirSin} and Theorem \ref{thm:ErrHomDirSin2DCoulomb}.  For the FS-MM method, the error mainly comes from the discretization of Fourier integral using the midpoint rule. Rigorous error analysis are provided in Theorem \ref{thm:ErrMidFou}.

 To quantify the error estimates, we define the following norm
\beas
 \| \varphi - \widetilde{\varphi} \|_{\infty(\Omega)} := \sup_{\bx \in \Omega} \left| (\varphi - \widetilde{\varphi}) (\bx) \right|,
\eeas
and we use $A \lesssim B$ to denote $A\leq c B$ where  $c>0$ is a constant.
For simplification of notations, below we omit the 2 in the symbols of $\rho_2$ and $\varphi_2$.

\vspace{0.5em}

\textbf{For the SP-MM method}, we aim to derive the error $\| \varphi - \widetilde{\varphi} \|_{\infty(\Omega)}$, where
$\varphi = U \ast \rho$ and $\widetilde{\varphi}$
satisfies the differential equation
\bea
\left\{\ba{ll}
 \mathcal{L} \widetilde{\varphi}(\bx)=\rho(\bx), \qquad \bx \in \Omega, \\[0.6em]
\widetilde{\varphi}(\bx) \big|_{\bx \in \partial \Omega}=0.
\ea\right.
\label{PDEHomDiri}
\eea
If the maximum norm estimate holds for the equation \eqref{PDEHomDiri}, we have the following theorem.

\vspace{0.5em}

\begin{thm}[SP-MM method]
 \label{thm:ErrHomDirSin}
For a smooth and compactly supported function $ \rho(\bx)$ that satisfies the vanishing moment condition
\eqref{MomVan1},
the following estimate
\begin{gather*}
\| \varphi- \widetilde{\varphi} \|_{\infty(\Omega)} \lesssim
\sum_{|\bm{\beta}| = m+1}
  \max_{\begin{subarray}{c}
   \bx \in \partial \Omega  \\
   \by \in  \Omega_0 \\
    0 \le s \le  1
  \end{subarray}} \left| {\partial^{\beta} U(\bx-s \by)  }\right|
\end{gather*}
holds true, where $\Omega_0:=\supp\{\rho\}   \subset \Omega$.
\end{thm}

\vspace{0.5em}

\begin{proof}
Let $w=\varphi-\widetilde{\varphi}$, using $ \mathcal{L} \varphi = \rho$ with $\bx \in \mathbb{R}^d$ and equation \eqref{PDEHomDiri},  we obtain
\beas
\left\{\ba{ll}
\mathcal{L} w(\bx)=0, \qquad ~~ \bx \in \Omega, \\[0.8em]
w (\bx)\big|_{\bx \in \partial\Omega}  = \varphi(\bx)\big|_{\bx \in \partial\Omega}.
\ea\right.
\label{ErrPDE}
\eeas
Using  maximum norm estimate \cite{MaximumPrinc}, we have
\beas
\| w \|_{\infty( \Omega)} \le \| w\|_{\infty(\partial \Omega)}= \| \varphi \|_{\infty(\partial \Omega)}.
\label{MaxNorm}
\eeas
With the help of Lemma \ref{lem:farAsyAna}, taking supremum with respect to $\bx$ for inequation \eqref{FarAsymMax}, we obtain
\begin{gather*}
\| \varphi - \widetilde{\varphi} \|_{\infty(\Omega)} \le \| \varphi \|_{\infty(\partial \Omega)} \lesssim
\sum_{|\bm{\beta}| = m+1}
  \max_{\begin{subarray}{c}
   \bx \in \partial \Omega  \\
   \by \in  \Omega_0 \\
    0 \le s \le  1
  \end{subarray}} \left| {\partial^{\beta} U(\bx-s \by)  }\right|.
\end{gather*}
Then the proof is completed.
\end{proof}
\vspace{0.5em}
\begin{remark}
For 3D Poisson potential with $U(\bx) = \frac{1}{4 \pi |
\bx|}$, we have
\begin{gather*}
\| \varphi - \tilde{\varphi} \|_{\infty(\Omega)} \lesssim
  \sum_{|\bm{\beta}| = m+1}
  \max_{\begin{subarray}{c}
   \bx \in \partial \Omega  \\
   \by \in  \Omega_0 \\
    0 \le s \le  1
  \end{subarray}} \frac{1}{\left|\bx-s \by \right|^{m+2}}
   \lesssim  \frac{1}{(L-L_0)^{m+2}},
\end{gather*}
where $\Omega_0 = [-L_0, L_0]^3$ and $\Omega = [-L, L]^3$.
\end{remark}
\vspace{0.5em}
\vspace{0.5em}

If the maximum norm estimate is not well-established, such as the 2D Coulomb potential, we derive the error estimates from a different perspective.
For ease of explanation, we take the 2D Coulomb potential as an example, and error estimates for other potentials can be derived using the same technique.
The numerical scheme is given by
\beas
\varphi_{N}(x, y)=\sum_{p=1}^{N-1} \sum_{q=1}^{N-1} \frac{ \hat{\rho}^s_{pq}}{\sqrt{\mu_p^2+\mu_q^2}} \sin(\mu_p (x+L)) \sin(\mu_q (y+L)),
\eeas
where
\beas
\hat{\rho}^s_{pq} = \frac{1}{L^2} \int_{\bR_L} \rho(x, y) \sin(\mu_p(x+L))\sin(\mu_q(y+L)) {\rm d}x {\rm d}y.
\eeas

\begin{thm}[SP-MM method for 2D Coulomb potential]
\label{thm:ErrHomDirSin2DCoulomb}
For a smooth and compactly supported function $ \rho(\bx)$ that satisfies the vanishing moment condition
\eqref{MomVan1},
the following estimate
 \beas
\| \varphi -\varphi_N\|_{\infty(\Omega)} \lesssim L^{-(m+2)}
\eeas
holds true for 2D Coulomb potential.
\end{thm}

\vspace{0.5em}

\begin{proof}
Let $w:=\varphi-\varphi_N$, we obtain
\bea
\left\{\ba{ll}
-\Delta w(\bx) = (-\Delta \varphi(\bx)) - (-\Delta \varphi_N(\bx)) := g(\bx), \qquad ~~ \bx \in \bR_L, \\[0.8em]
w (\bx)\big|_{\bx \in \partial\bR_L}  = \varphi(\bx)\big|_{\bx \in \partial\bR_L}.
\ea\right.
\label{ErrPDECou}
\eea
Using the maximum norm estimate, we have
\bea
\| w \|_{\infty( \bR_L)} \le \| \varphi\|_{\infty(\partial \bR_L)}+ \| g\|_{\infty(\bR_L)}.
\label{maxNormw}
\eea
With the help of the Lemma \ref{lem:farAsyAna}, we have
\bea
\|\varphi\|_{\infty(\partial \bR_L)} \lesssim L^{-(m+2)}.
\label{errBound}
\eea
Next, we focus on estimating
 $\| g\|_{\infty(\bR_L)}$. In fact, $g(\bx)$ can be equivalently rewritten as
 \bea
 g(x, y) =&& \frac{1}{(2 \pi)^2}\int_{\mathbb{R}^2}
 \sqrt{k_1^2 + k_2^2}~ \widehat{\rho}(k_1, k_2) e^{\im (k_1 x + k_2 y)} {\rm d} k_1 {\rm d} k_2- \notag \\ &&\sum_{p=1}^{N-1} \sum_{q=1}^{N-1}\sqrt{\mu_p^2+\mu_q^2} ~ \hat{\rho}^s_{pq} \sin(\mu_p (x+L)) \sin(\mu_q (y+L)) .
  \label{disErrCou}
 \eea
 Our idea is to utilize the relationship
 \bea
  \sin(x)=(e^{\im x}-e^{-\im x})/(2 \im)
  \label{Relasinexp}
  \eea
  to transform $sin(x)$ function in \eqref{disErrCou} into $e^{\im x}$ function. Specifically,
 a tedious computation shows that
 \beas
-\Delta \varphi_N &=&\sum_{p=1}^{N-1} \sum_{q=1}^{N-1}\sqrt{\mu_p^2+\mu_q^2} ~ \hat{\rho}^s_{pq} \Big(\frac{e^{\im \mu_p (x+L)} -e^{-\im \mu_p (x+L)} }{2 \im} \Big)
\Big(\frac{e^{\im \mu_q (y+L)} -e^{-\im \mu_q (y+L)} }{2 \im} \Big) \\
&=&-\frac{1}{4}\sum_{p=1}^{N-1} \sum_{q=1}^{N-1}\sqrt{\mu_p^2+\mu_q^2} ~ \hat{\rho}^s_{pq} \bigg(e^{\im \big[\mu_p (x+L)+ \mu_q (y+L)\big]}
-e^{\im \big[\mu_p (x+L)- \mu_q (y+L)\big]} \\
&&\qquad \qquad \qquad  \qquad \qquad ~ -e^{\im \big[-\mu_p (x+L)+ \mu_q (y+L)\big]}
+e^{-\im \big[\mu_p (x+L)+ \mu_q (y+L)\big]}
\bigg).
 \eeas
 Noted that $\hat{\rho}_{p(-q)}^s = -\hat{\rho}_{pq}^s$, we have
 \beas
 &&\sum_{p=1}^{N-1} \sum_{q=1}^{N-1}\sqrt{\mu_p^2+\mu_q^2} ~ \hat{\rho}^s_{pq} e^{\im \big[\mu_p (x+L)- \mu_q (y+L)\big]} \\
 = &&\sum_{p=1}^{N-1} \sum_{q=-N+1}^{-1}\sqrt{\mu_p^2+\mu_q^2} ~ (-\hat{\rho}^s_{pq})~ e^{\im \big[\mu_p (x+L)+\mu_q (y+L)\big]}
 \eeas
 In a similar way, we have
 \bea
-\Delta \varphi_N &=&-\frac{1}{4}
\sum_{\substack { p=-N+1\\ p  \not=0}  }^{N-1} \sum_{\substack { q=-N+1\\ q  \not=0}  }^{N-1}\sqrt{\mu_p^2+\mu_q^2} ~ \hat{\rho}^s_{pq}~ e^{\im (\mu_p +\mu_q)L}~ e^{\im (\mu_p x+\mu_q y)}.
\label{LapPhiN1}
 \eea
Next, we aim to derive the relationship between $\hat{\rho}^s$ and $\widehat{\rho}$.
Using Eqn.  \eqref{Relasinexp}, we have
\beas
\hat{\rho}^s_{pq}
&=&
-\frac{1}{4 L^2} \int_{\bR_L} \rho(x, y)
\bigg(e^{\im \big[\mu_p (x+L)+ \mu_q (y+L)\big]}
-e^{\im \big[\mu_p (x+L)- \mu_q (y+L)\big]} \\
&&\qquad \qquad \qquad  \quad ~ -e^{\im \big[-\mu_p (x+L)+ \mu_q (y+L)\big]}
+e^{-\im \big[\mu_p (x+L)+ \mu_q (y+L)\big]}
\bigg) {\rm d} x {\rm d} y.
\eeas
For the first term,  we apply the change of variables
$x \mapsto -x$ and $y \mapsto -y$ in the integral to obtain
\beas
&&\int_{\bR_L}\rho(x, y) e^{\im \big[\mu_p (x+L)+ \mu_q (y+L)\big]} {\rm d} x {\rm d} y\\
=&& \int_{\bR_L}\rho(-x, -y) e^{\im \big[\mu_p (-x+L)+ \mu_q (-y+L)\big]} {\rm d} x {\rm d} y
\\
= &&\int_{\bR_L}(-1)^{p+q}\rho(-x, -y) e^{-\im \big[\mu_p (x+L)+ \mu_q (y+L)\big]} {\rm d} x {\rm d} y.
\eeas
Therefore, we obtain
\beas
\hat{\rho}^s_{pq} =
-\frac{1}{4 L^2} \int_{\bR_L} \Big[&&(-1)^{p+q}\rho(-x, -y)-
(-1)^{p}\rho(-x, y)-(-1)^{q}\rho(x, -y)+
\rho(x, y)
\Big] \\
&&e^{-\im \big[\mu_p (x+L)+ \mu_q (y+L)\big]}
 {\rm d} x {\rm d} y.
\eeas
If the density $\rho$ is an odd or even function, the above expression can be further simplified. However, the density $\rho$ does not generally exhibit such parity properties. Fortunately, any univariate function can be decomposed into the sum of an even and an odd function. Thus, we prove the result separately for the following four cases.

\vspace{0.5em}

{\bf{Case} 1:} When $\rho(x,y)$ is even with respect to both
$x$ and $y$, i.e., $\rho(-x,y) = \rho(x,y)$ and $\rho(x,-y) = \rho(x, y)$,
we have
\bea
\hat{\rho}^s_{pq} &=& -\frac{1}{4 L^2} e^{-\im( \mu_p+ \mu_q)L} \big[1-(-1)^p \big]\big[1-(-1)^q \big] \int_{\bR_L} \rho(x, y) e^{-\im (\mu_p x+ \mu_q y)}
 {\rm d} x {\rm d} y \notag \\[0.8em]
 &=&
 \left\{\ba{ll}
-\frac{1}{L^2}  e^{-\im( \mu_p+ \mu_q)L}
 ~ \widehat{\rho}(\mu_p, \mu_q),  \qquad p ~ \text{odd} ~ \& ~ q ~ \text{odd}, \\[0.8em]
0, \qquad \qquad \qquad \qquad \qquad \qquad \qquad \text{else}.
\ea\right.
\label{rhohatRela}
\eea
Plugging \eqref{rhohatRela} into \eqref{LapPhiN1}, $-\Delta \varphi_N$ can be simplified as
\beas
-\Delta \varphi_N &=&\frac{1}{4L^2}
\sum_{p=-N/2}^{N/2-1} \sum_{q=-N/2}^{N/2-1}\sqrt{\mu_{2p+1}^2+\mu_{2q+1}^2} ~ \widehat{\rho}(\mu_{2p+1}, \mu_{2 q+1})~ e^{\im (\mu_{2p+1} x+\mu_{2q+1} y)} \\[0.5em]
&=&
\frac{1}{4L^2}
\sum_{p=-N/2}^{N/2-1} \sum_{q=-N/2}^{N/2-1}\sqrt{\nu_{p+\frac{1}{2}}^2+\nu_{q+\frac{1}{2}}^2} ~ \widehat{\rho}(\nu_{p+\frac{1}{2}}, \nu_{q+\frac{1}{2}})~ e^{\im \big(\nu_{p+\frac{1}{2}} x+\nu_{q+\frac{1}{2}} y\big)},\\
&:=& \Big(\frac{2 \pi}{2L}\Big)^2 \sum_{p=-N/2}^{N/2-1} \sum_{q=-N/2}^{N/2-1} f(\nu_{p+\frac{1}{2}}, \nu_{q+\frac{1}{2}}),
\eeas
where $f(k_1, k_2) := \frac{1}{ (2\pi)^2}\sqrt{k_1^2 +k_2^2} ~\widehat{\rho}(k_1, k_2)~ e^{\im (k_1 x + k_2 y)}$ ,  $\nu_p = (2 \pi p )/(2L)$ and $\nu_q = (2 \pi q)/(2L) $.
Note that $f$ depends on $x$ and $y$.
For simplicity, we omit the explicit dependence on
$x$ and $y$ in the notation.
 Since $\widehat{\rho}$ is compactly supported, $f$ is also compactly supported.
Therefore, it is reasonable to choose a fixed mesh size $h$ such that $\supp \{\widehat{f} \} \subsetneq  R_h := [-\pi/h, \pi/h]^2 $.
In this case, the function $g$ in \eqref{disErrCou} equivalently becomes
\beas
g &=& \int_{R_h}
 f(k_1, k_2) {\rm d} k_1 {\rm d} k_2 -  \Big(\frac{2 \pi}{2L}\Big)^2  \sum_{p=-N/2}^{N/2-1} \sum_{q=-N/2}^{N/2-1} f(\nu_{p+\frac{1}{2}}, \nu_{q+\frac{1}{2}}) \\
 &:=& \int_{R_h}
 f(k_1, k_2) {\rm d} k_1 {\rm d} k_2 - Q^{ \rm{MM}}(f) := E^{\rm{MM}}(f).
\eeas
Clearly, $g$ is the discretization error of the midpoint rule.
We consider the Fourier series expansion of $f$ and denote
\beas
P_K(f)(k_1, k_2):= \sum_{p=-K}^{K} \sum_{q=-K}^K \widehat{f}_{pq} ~ e^{\im (\lambda_p k_1 + \lambda_q k_2)},
\eeas
where $\lambda_p = \frac{2 \pi}  {2 \pi/h} p = p h$ and the Fourier coefficients are given as follows
 \beas
 \hat{f}_{pq} &=& \frac{1}{|R_h|} \int_{R_h} f(k_1, k_2) e^{-\im (\lambda_p k_1+ \lambda_q k_2)} {\rm d}k_1 {\rm d}k_2 \\
 &=&\frac{1}{|R_h|} \frac{1}{(2\pi)^2} \int_{R_h}  \sqrt{k_1^2 +k_2^2} ~\widehat{\rho}(k_1, k_2)~ e^{\im \big[k_1 (x-\lambda_p) + k_2 (y-\lambda_q)\big]}
 {\rm d}k_1 {\rm d}k_2 \\
  &=&  -\frac{h^2}{4 \pi^2} \Delta \varphi(x-\lambda_p, y-\lambda_q).
 \eeas
For simplicity, we define $\bp  = (p, q)$ and
$\mathcal{T}_{K}: = \{(p, q) \in \mathbb{Z}^2,~ p= -K \dots, K, q= -K, \dots, K \}$.
By leveraging continuity and piecewise differentiability of $f$, along with the following estimates
\beas
 && \Big| \sum_{(p, q)  \notin \mathcal{T}_K} \widehat{f}_{pq} e^{\im (\lambda_p k_1 + \lambda_q k_2)}\Big| \le \sum_{(p, q) \notin \mathcal{T}_K} \left| \widehat{f}_{pq} \right|
   \lesssim \sum_{|\bm p|>K} \left| \widehat{f}_{\bm p} \right|  \\
   &&
  \lesssim  \sum_{|\bm p|>K}  \left|\Delta \varphi(\bx- \bm p h )\right| \lesssim \sum_{|\bm p|>K} \frac{1}{|\bm p h - \bx|^{(m+4)}}
\lesssim \sum_{|\bm p|>K} \frac{1}{|\bm p- \frac{\bx}{h}|^{m+4}} \\
&& \lesssim \int_{K}^{\infty} \frac{1}{\big||\bx|-\frac{\sqrt{2} L}{h}\big|^{m+4}} {\rm d}{\bx} = 2 \pi
\int_{K}^{\infty} \frac{r }{\big|r-\frac{\sqrt{2} L}{h}\big|^{m+4}} {\rm d}{r} \\
&&\lesssim \frac{1}{K^{m+2}}  \rightarrow 0,  \qquad K \rightarrow + \infty,
\eeas
we obtain that $f$  has a uniformly convergent Fourier series, i.e.,
\beas
f= \sum_{p=-\infty}^{\infty} \sum_{q=-\infty}^\infty \widehat{f}_{pq} ~ e^{\im (\lambda_p k_1 + \lambda_q k_2)}.
\eeas
Using the linearity of $E^{\rm{MM}}(f)$ and the uniform convergence of the Fourier series of $f$, we have
\bea
E^{\rm{MM}}(f) = \sum_{p=-\infty}^{\infty} \sum_{q=-\infty}^{\infty} \hat{f}_{pq}~ E^{\rm{MM}} \left(e^{\im (\lambda_p k_1+ \lambda_q k_2)}\right).
\label{ErrMidF}
\eea
A simple calculation shows that
\bea
E^{\rm{MM}}(e^{\im (\lambda_p k_1+ \lambda_q k_2)}) =
 \left\{\ba{ll}
-\left(\frac{2 \pi}{h}\right)^2 (-1)^{p+q}, \quad \quad p,~ q \in N\mathbb{N}^{+}, \\[0.8em]
0, \qquad \qquad \qquad \qquad ~~\quad  \text{else}.
\ea\right.
\label{ErrMidExp}
\eea
By plugging Eqn. \eqref{ErrMidExp} into Eqn. \eqref{ErrMidF} and using Lemma \ref{lem:farAsyAna},
we have
\beas
|g| &=& \left| E^{\rm{MM}} (f) \right|= \Big| -\left(\frac{2 \pi}{h}\right)^2 \sum_{|p|>0}  \sum_{|q|>0} (-1)^{(p+ q)N}~ \hat{f}_{pN, qN} \Big| \notag \\
& \lesssim &\sum_{|p|>0} \sum_{|q|>0} \Big|  \Delta \varphi(x-\lambda_{pN}, y-\lambda_{qN}) \Big|
\lesssim \sum_{|\bp|>0} \frac{1}{|\bx-2 L \bp|^{m+4}} \notag\\
& = & \frac{1}{L^{m+4}} \sum_{|\bp|>0} \frac{1}{|\bp-\frac{\bx}{2L}|^{m+4}}
\lesssim \frac{1}{L^{m+4}}.
 \eeas

  \vspace{0.5em}

  {\bf{Case} 2:} When $\rho(x,y)$ is even with respect to
$x$ and odd with respect to $y$, i.e., $\rho(-x,y) = \rho(x,y)$ and $\rho(x,-y) = -\rho(x, y)$,
we have
\bea
\hat{\rho}^s_{pq} &=& -\frac{1}{4 L^2} e^{-\im( \mu_p+ \mu_q)L} \big[1-(-1)^p \big]\big[1+(-1)^q \big] \int_{\bR_L} \rho(x, y) e^{-\im (\mu_p x+ \mu_q y)}
 {\rm d} x {\rm d} y \notag \\[0.8em]
 &=&
 \left\{\ba{ll}
-\frac{1}{L^2}  e^{-\im( \mu_p+ \mu_q)L}
 ~ \widehat{\rho}(\mu_p, \mu_q),  \qquad p ~ \text{odd} ~ \& ~ q ~ \text{even}, \\[0.8em]
0, \qquad \qquad \qquad \qquad \qquad \qquad \qquad \text{else}.
\ea\right.
\label{rhohatRelaevod}
\eea
Similarly, by plugging \eqref{rhohatRelaevod} into \eqref{LapPhiN1}, $-\Delta \varphi_N$ can be further simplified as
 \beas
 -\Delta \varphi_N =
 \Big(\frac{2 \pi}{2L}\Big)^2 \sum_{p=-N/2}^{N/2-1} \sum_{\substack { q=-N/2+1\\ q  \not=0} }^{N/2-1} f(\nu_{p+\frac{1}{2}}, \nu_{q})
 = \Big(\frac{2 \pi}{2L}\Big)^2 \sum_{p=-N/2}^{N/2-1} \sum_{q=-N/2}^{N/2-1} f(\nu_{p+\frac{1}{2}}, \nu_{q}),
 \eeas
 where the second equality holds because
 $f(\nu_{p+\frac{1}{2}}, \nu_0) = f(\nu_{p+\frac{1}{2}}, \nu_{-N/2})  = 0$.
Then the function $g$ in \eqref{disErrCou} equivalently becomes
\beas
g &=& \int_{R_h}
 f(k_1, k_2) {\rm d} k_1 {\rm d} k_2 -  \Big(\frac{2 \pi}{2L}\Big)^2  \sum_{p=-N/2}^{N/2-1} \sum_{q=-N/2}^{N/2-1} f(\nu_{p+\frac{1}{2}}, \nu_{q}) \\
 &:=& \int_{R_h}
 f(k_1, k_2) {\rm d} k_1 {\rm d} k_2 -
  Q^{\rm{ML}}(f)
  = E^{\rm{ML}}(f).
\eeas
Similar to Case 1, we have
\beas
\big|g \big| &=& \big|E^{\rm{ML}}(f) \big| = \Big| \sum_{p=-\infty}^{\infty} \sum_{q=-\infty}^{\infty} \hat{f}_{pq}~ E^{\rm{ML}} (e^{\im (\lambda_p k_1+ \lambda_q k_2)})\Big| \notag \\
&=&
\Big| -\left(\frac{2 \pi}{h}\right)^2 \sum_{|p|>0}  \sum_{|q|>0} (-1)^{pN}~ \hat{f}_{pN, qN} \Big|
 \lesssim \sum_{|p|>0} \sum_{|q|>0} \Big|  \Delta \varphi(x-\lambda_{pN}, y-\lambda_{qN}) \Big| \notag  \\
&\lesssim& \sum_{|\bp|>0} \frac{1}{|\bx-2 L \bp|^{m+4}}
\lesssim \frac{1}{L^{m+4}}.
\eeas

\vspace{0.5em}

 {\bf{Case} 3:} When $\rho(x,y)$ is odd with respect to
$x$ and even with respect to $y$, i.e., $\rho(-x,y) = -\rho(x,y)$ and $\rho(x,-y) = \rho(x, y)$, the proof is analogous to Case 1 and Case 2. For brevity, the detailed steps are omitted.

\vspace{1em}

 {\bf{Case} 4:} When $\rho(x,y)$ is odd with respect to
both $x$ and $y$, i.e., $\rho(-x,y) = -\rho(x,y)$ and $\rho(x,-y) = -\rho(x, y)$, the proof is analogous to Case 1 and Case 2. For brevity, the detailed steps are omitted.

\vspace{1em}

For general density $\rho$, we  perform an odd-even decomposition, i.e.,
\bea
\label{oddEvenDeco}
\rho(x,y) = \rho^{ee}(x,y) + \rho^{eo}(x,y)+\rho^{oe}(x,y)+\rho^{oo}(x,y),
\eea
where
\beas
\rho^{ee}(x,y) &=& \frac{1}{4} \Big[
\rho(x, y) + \rho(x, -y)+\rho(-x, y)+\rho(-x, -y)
 \Big],  \\
 \rho^{eo}(x,y) &=& \frac{1}{4} \Big[
\rho(x, y) - \rho(x, -y)+\rho(-x, y)-\rho(-x, -y)
 \Big], \\
  \rho^{oe}(x,y) &=& \frac{1}{4} \Big[
\rho(x, y) + \rho(x, -y)-\rho(-x, y)-\rho(-x, -y)
 \Big], \\
   \rho^{oo}(x,y) &=& \frac{1}{4} \Big[
\rho(x, y) - \rho(x, -y)-\rho(-x, y)+\rho(-x, -y)
 \Big].
\eeas
Importantly, the function $\rho^{\alpha \beta}$ with $\alpha, \beta \in \{e, o\}$ retains the properties of smoothness, compact support and moment-matching \eqref{MomVan1}, as dose density $\rho$.
Based on the discussion in Case 1 - Case 4, we have
\bea
|g| &=& \big|E^{\rm{MM}}(f^{ee})+
E^{\rm{ML}}(f^{eo})+
E^{\rm{LM}}(f^{oe}) +
E^{\rm{LL}}(f^{oo}) \big| \notag \\
&\le &
\big|E^{\rm{MM}}(f^{ee})\big|+
\big|E^{\rm{ML}}(f^{eo})\big|+
\big|E^{\rm{LM}}(f^{oe})\big|+
\big|E^{\rm{LL}}(f^{oo})\big|
\notag \\
& \lesssim & \frac{1}{L^{m+4}}.
\label{errdiscre}
\eea
 Therefore, the proof is completed by combing
  \eqref{maxNormw}, \eqref{errBound} and \eqref{errdiscre}.
\end{proof}

\vspace{1em}

\textbf{For the FS-MM method}, our goal is to estimate the error $\| \varphi - \varphi_N\|_{\infty(\Omega)}$, where
$\varphi = U \ast \rho$ and $\varphi_N$ is given by \eqref{SchemeMid}. In fact,
the error primarily comes from the discretization of the Fourier integral via the midpoint  quadrature. We present the following theorem for its rigorous analysis.

\vspace{0.5em}

\begin{thm}[FS-MM method]
\label{thm:ErrMidFou}
For a smooth and compactly supported function $ \rho(\bx)$ that satisfies the vanishing moment condition
\eqref{MomVan1},
the following estimate
\begin{gather*}
\| \varphi- \varphi_N \|_{\infty(\Omega)} \lesssim
\sum_{|\bm{\beta}| = m+1}
  \max_{\begin{subarray}{c}
   \bx \in \partial \Omega  \\
   \by \in  \Omega_0 \\
    0 \le s \le  1
  \end{subarray}} \left| {\partial^{\beta} U(\bx-s \by)  }\right|
\end{gather*}
holds true
where
$\Omega_0:=\supp\{\rho\}   \subset \Omega$.
\end{thm}

\vspace{0.5em}

\begin{proof}
Since $\widehat{\rho}$ is compactly supported,
 it is reasonable to choose a fixed $h$ such that $\supp \{\widehat{\rho} \} \subsetneq  R_h := [-\pi/h, \pi/h]^d $.
 Therefore, the error comes from the discretization of the integral using the midpoint quadrature, i.e.,
 \beas
 \varphi-\varphi_N = I(f) -  Q^{\rm{M}}_{\rm{N}}(f)=
 E^{\rm{M}}_{\rm{N}}(f),
 \eeas
where $f(\bk) =\widehat{U}(\bk) \widehat{\rho}(\bk) e^{\im \bk \cdot \bx}$,
$I(f)$ represents the integral of the function $f$ over $R_h$, $Q^{\rm{M}}_{\rm{N}}(f)$ denotes the midpoint quadrature with $N$ points in each direction and $E^{\rm{M}}_{\rm{N}}(f)$ represents the corresponding error.

Similar to the proof in Theorem \ref{thm:ErrHomDirSin2DCoulomb}, by taking the 2D case as an example, we can obtain
    \beas
  && \big| E^{\rm{MM}}_{\rm{NN}}(f^{ee})\big|
   \lesssim
  \max_{\begin{subarray}{c}
   \bx \in \partial \Omega  \\
   \by \in  \Omega_0 \\
    0 \le s \le  1
  \end{subarray}} \left| {\partial^{\beta} U(\bx-s \by)  }\right|,  \quad
   \big| E^{\rm{ML}}_{\rm{NN}}(f^{eo})\big|
   \lesssim
  \max_{\begin{subarray}{c}
   \bx \in \partial \Omega  \\
   \by \in  \Omega_0 \\
    0 \le s \le  1
  \end{subarray}} \left| {\partial^{\beta} U(\bx-s \by)  }\right|, \\
  && \big| E^{\rm{LM}}_{\rm{NN}}(f^{oe})\big|
   \lesssim
  \max_{\begin{subarray}{c}
   \bx \in \partial \Omega  \\
   \by \in  \Omega_0 \\
    0 \le s \le  1
  \end{subarray}} \left| {\partial^{\beta} U(\bx-s \by)  }\right|, \quad
   \big| E^{\rm{LL}}_{\rm{NN}}(f^{oo})\big|
   \lesssim
  \max_{\begin{subarray}{c}
   \bx \in \partial \Omega  \\
   \by \in  \Omega_0 \\
    0 \le s \le  1
  \end{subarray}} \left| {\partial^{\beta} U(\bx-s \by)  }\right|.
    \eeas
 With the help of the relation between midpoint rule and left rectangle rule, i.e.,
 \beas
 Q^{\rm{MM}}_{\rm{NN}}(f) +  Q^{\rm{ML}}_{\rm{NN}}(f) =
 2  Q^{\rm{ML}}_{\rm{N(2N)}}(f),
 \eeas
 we have
 \beas
 E^{\rm{MM}}_{\rm{NN}}(f) &=& I(f) - Q^{\rm{MM}}_{\rm{NN}}(f)
  = I(f) - \Big[
 2  Q^{\rm{ML}}_{\rm{N (2N)}}(f) -Q^{\rm{ML}}_{\rm{NN}}(f) \Big] \\
 &=& 2 \Big[ I(f) -  Q^{\rm{ML}}_{\rm{N (2N)}}(f)  \Big]
 -\Big[I(f) -Q^{\rm{ML}}_{\rm{NN}}(f)  \Big] \\
 &=& 2 E^{\rm{ML}}_{\rm{N (2N)}}(f) -
 E^{\rm{ML}}_{\rm{N N}}(f).
 \eeas
 In a similar way, we obtain
 \beas
  E^{\rm{MM}}_{\rm{N N}}(f) =  2 E^{\rm{LM}}_{\rm{(2N)N}}(f)-
 E^{\rm{LM}}_{\rm{N N}}(f), \qquad
  E^{\rm{ML}}_{\rm{N N}}(f) =  2 E^{\rm{LL}}_{\rm{(2N) N}}(f)-
 E^{\rm{LL}}_{\rm{N N}}(f).
 \eeas
 Therefore, we have
 \beas
 \big| \varphi-\varphi_N \big| =
 \big|  E^{\rm{MM}}_{\rm{N N}}(f) \big|
 &\le&
 \big|  E^{\rm{MM}}_{\rm{N N}}(f^{ee}) \big| +
\big|  E^{\rm{MM}}_{\rm{N N}}(f^{eo}) \big| +
\big|  E^{\rm{MM}}_{\rm{N N}}(f^{oe}) \big| +
\big| E^{\rm{MM}}_{\rm{N N}}(f^{oo}) \big| \\
&\lesssim&
\big| E^{\rm{MM}}_{\rm{N N}}(f^{ee}) \big| +
\big|  E^{\rm{ML}}_{\rm{N N}}(f^{eo}) \big| +
\big| E^{\rm{LM}}_{\rm{N N}}(f^{oe}) \big| +
\big|  E^{\rm{LL}}_{\rm{N N}}(f^{oo}) \big|  \\
&\lesssim &
\sum_{|\bm{\beta}| = m+1}
\max_{\begin{subarray}{c}
   \bx \in \partial \Omega  \\
   \by \in  \Omega_0 \\
    0 \le s \le  1
  \end{subarray}} \left| {\partial^{\beta} U(\bx-s \by)  }\right|.
 \eeas
 Then the proof is completed.
\end{proof}

\section{Numerical results} \label{sec:NumRes}
In this section, we shall investigate the accuracy and efficiency for different nonlocal potentials in both 2D and 3D cases.  The computational domain $\bR_L$ is discretized uniformly in each spatial direction with mesh size $h_j$, and we define mesh size vector as $\bh = (h_1, \dots, h_d)$.
For simplicity, we shall use $h$ to denote the mesh size if all the mesh sizes are equal.
All numerical errors are calculated in the relative maximum norm, defined as follows
\begin{align*}
\mathcal{E}:
=\max_{\bx \in \mathcal T_{\bh}} |\varphi(\bx)-\varphi_{\bh}(\bx)|/\max_{\bx \in \mathcal T_{\bh}} |\varphi(\bx)|,
\end{align*}
where  $\varphi_{\bh}$ is the numerical solution on mesh grid $\mathcal{T}_{\bh}$ and $\varphi $ is the reference solution.

Unless otherwise specified,
 we use the auxiliary function \eqref{AuxilFunc} with $\sigma =2$ and the density functions
 $\rho(\bx) = e^{-|\bx-\bx_0|^2/4} $,
 where $\bx_0 = (1,2)$ for $d=2$ and
 $\bx_0 = (1, 2, 3)$ for $d=3$.
The solution obtained by KTM \cite{FastConvGreengard}  is taken as the reference solution.
 In practice, we choose  computational domain $\mathbf{R}_{L} = [-16,16]^d$ and mesh size $h=1/4$.

\subsection{The Poisson potentials in 2D/3D}

\begin{exmp}[\textbf{Accuracy test}]
\label{Exmp:PoissonAccu}
Here, we consider  the 2D and 3D Poisson potentials with convolution kernel
   \begin{equation*}
      U(\bx) =  \left\{
         \begin{array}{ll}
           -\frac{1}{2\pi} \ln(|\bx|),     & d=2,\\[0.5em]
            \frac{1}{4\pi |\bx|},   & d=3.
         \end{array}
      \right.
   \end{equation*}
\end{exmp}

 Figure \ref{Fig:PoissonAccu} shows the  errors and convergence orders with respect to the domain expansion factor $S$. From these results, we can see that the convergence order is at least $ m+1$ for the 2D Poisson potential and $m+2$ for the 3D Poisson potential.  It is compatible with theoretical error estimates and performs better for certain $m$.

\begin{figure}[!htpb]
\centering
\includegraphics[scale=0.44]{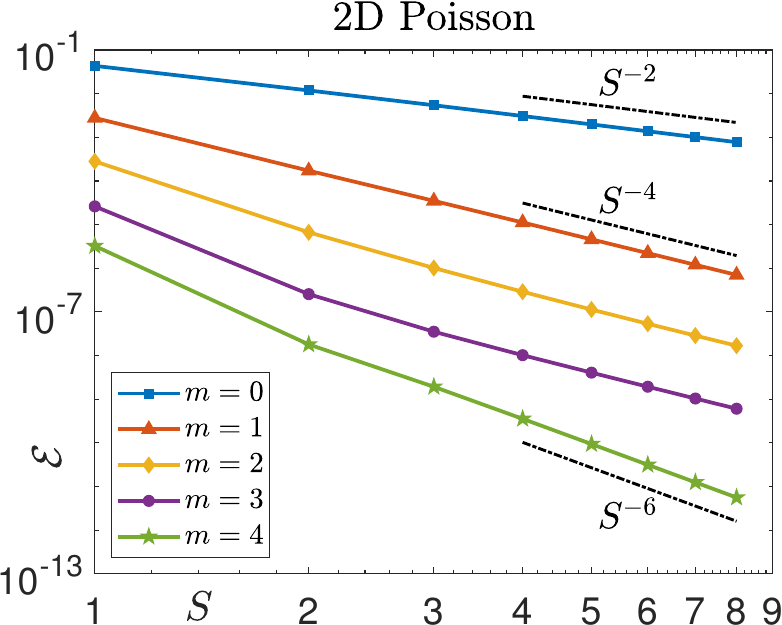}
\hspace{0.5em}
\includegraphics[scale=0.44]{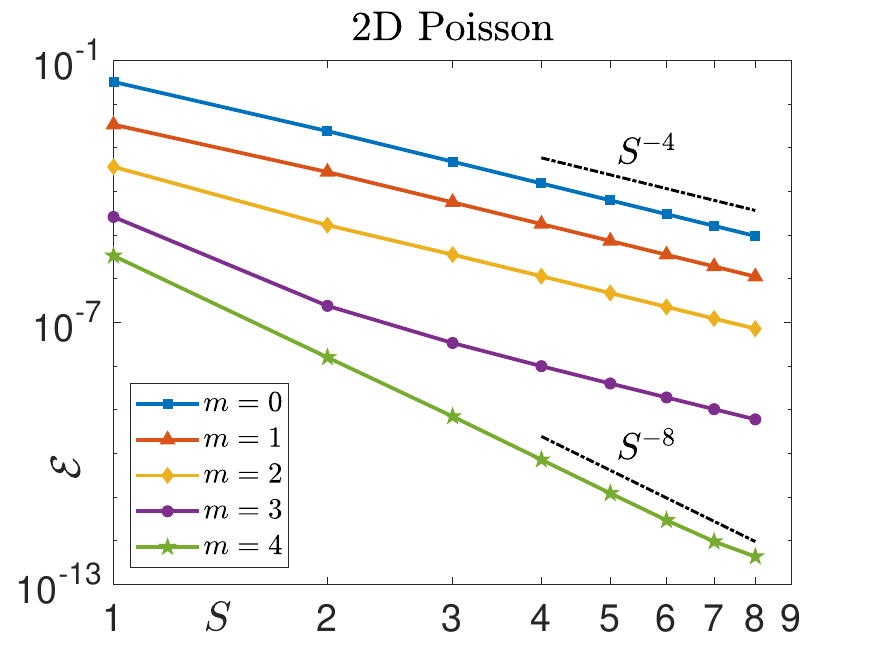} \\[0.5em]
\includegraphics[scale=0.44]{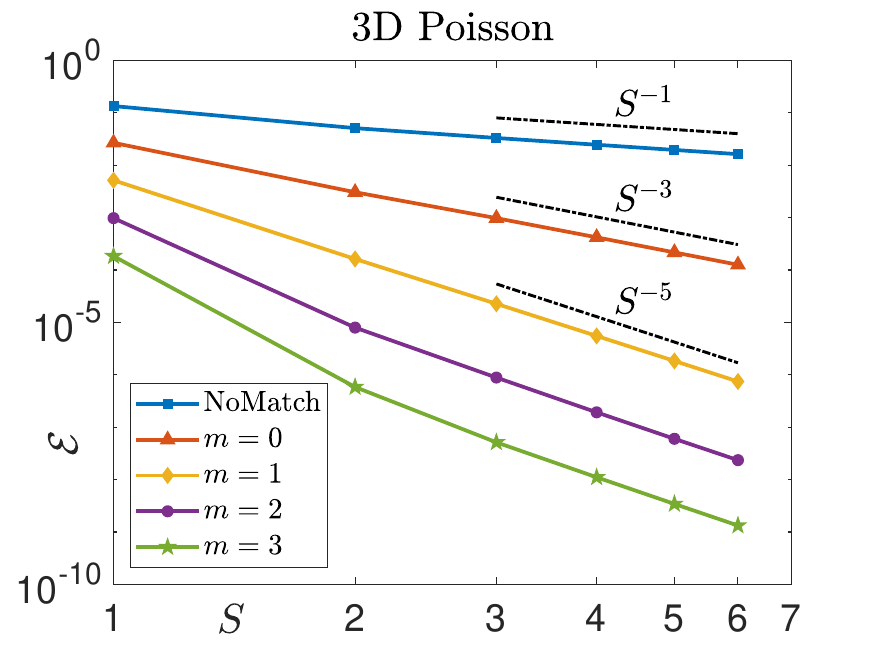} \hspace{-0.7em}
\includegraphics[scale=0.44]{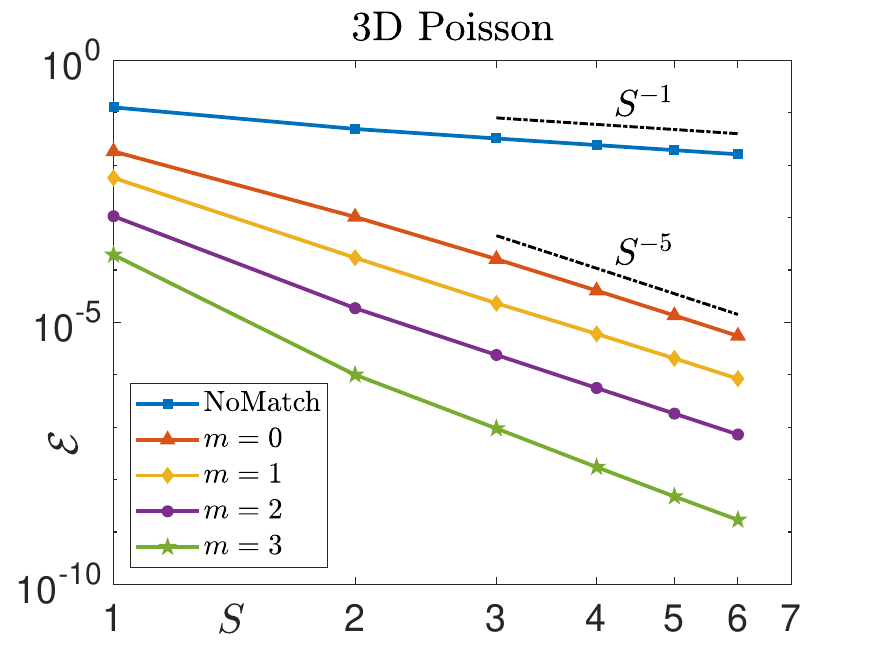} \\
\caption{Errors of the $\mathrm{SP\text{-}MM}$ (left) and  $\mathrm{FS\text{-}MM}$ (right) methods for Poisson potentials. }
\label{Fig:PoissonAccu}
\end{figure}

\begin{exmp}[\textbf{Efficiency test}]
\label{Exmp:PoissonEffi}
To investigate the efficiency, we present the performance of $\mathrm{SP\text{-}MM}$ and $\mathrm{FS\text{-}MM}$ methods with/without tensor acceleration in terms of CPU time, where we choose the 3D Poisson potentials with $m=2$.
The computation of tensor version is split into two parts: the precomputation part ({\bf PreComp}) and the execution part ({\bf Execution}).
The algorithms were implemented
in Matlab (2016a) and run on a 3.00GH Intel(R) Xeon(R) Gold 6248R CPU with a 36 MB cache in Ubuntu GNU/Linux.
\end{exmp}

\begin{figure}
\centering
\includegraphics[width=0.495\textwidth]{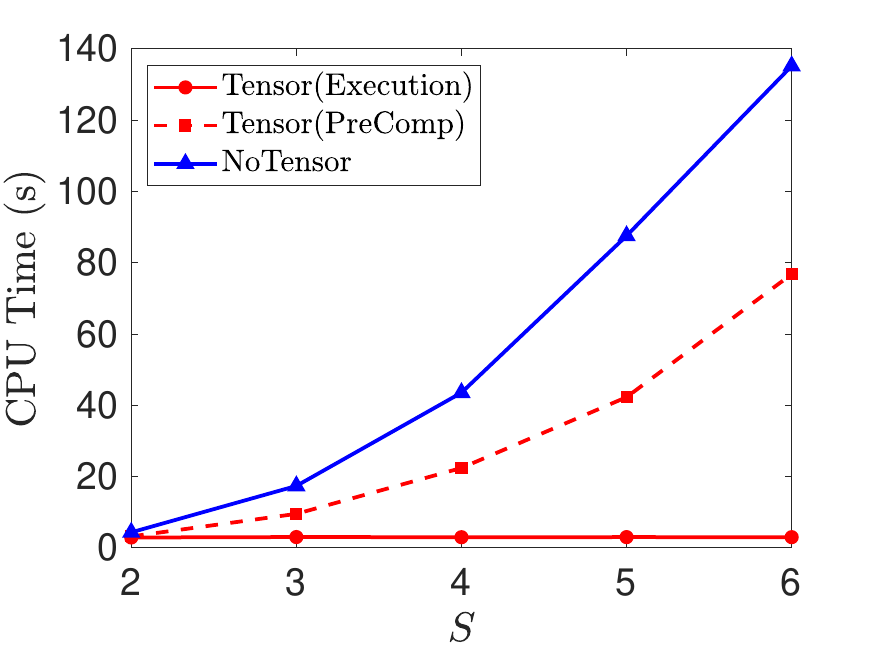}
\includegraphics[width=0.495\textwidth]{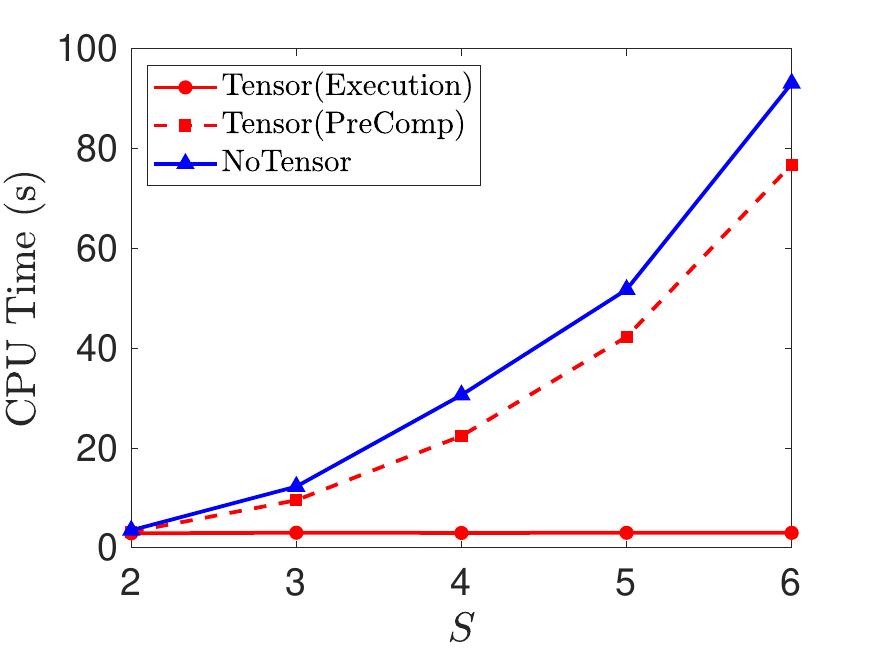}
\caption{Timing results of $\mathrm{SP\text{-}MM}$
 (left) and  $\mathrm{FS\text{-}MM}$ (right) methods, with/without tensor acceleration, versus domain expansion factor $S$. }
\label{TimeTest}
\end{figure}

In Figure \ref{TimeTest}, we present the CPU time for different domain expansion factor $S$.
We can conclude that, for  tensor version, once the precomputation step is done, the effective computation involves only a pair of FFT/iFFT on vectors of twice the length in each direction, regardless of the magnitude of domain expansion factor $S$.
This is of significant importance in practical simulation, especially when the potential evaluation needs to be performed multiple times under the same setups.

\subsection{The 2D Coulomb potentials and 3D DDI}
\begin{exmp}
\label{Exmp:CoulombAccu}
Here, we consider  the following potentials
  \begin{itemize}[itemsep=8pt, topsep=4pt]
  \item {\bf 2D Coulomb potential}:
  The kernel is given as $U(\bx) =\frac{1}{2 \pi |\bx|} $.
\item {\bf 3D DDI}:  The kernel is given as
\beas
 U(\mathbf{x})  =
  \frac{3}{4 \pi} \frac{\textbf{m} \cdot \textbf{n} - 3 (\mathbf{x} \cdot \textbf{m}) (\mathbf{x} \cdot \textbf{n})/|\mathbf{x}|^2 }{|\mathbf{x}|^3},
\eeas
where $\textbf{n}$, $\textbf{m} \in \mathbb{R}^3$ are unit vectors representing the dipole orientations, and
the 3D potential is reformulated as
\cite{BaoBECRev}:
  \beas\
  \varphi(\bx)= -(  \textbf{m} \cdot \textbf{n} ) \rho(\bx)-3 \frac{1}{4 \pi |\bx|} \ast (\partial_{\textbf{n} \textbf{m}} \rho),
\eeas
where $\partial_{\textbf{m}} = \textbf{m} \cdot  \nabla $ and $ \partial_{\textbf{n}\textbf{m}}= \partial_{\textbf{n}} (\partial_{\textbf{m}})$.
In fact, the potential can be calculated via the 3D Poisson potential with source term $\partial_{\textbf{n}\textbf{m}} \rho$,
which can be easily computed numerically via Fourier spectral method.
\end{itemize}
\end{exmp}

 Figure \ref{Fig:CoulombDDIAccu} presents the  errors and convergence orders with respect to the domain expansion factor $S$.
 The results indicate that the convergence order is at least $ m+2$, i.e., $\mathcal{O}(S^{-(m+2)})$, for both 2D Coulomb potential and 3D DDI.
For the 3D DDI, results obtained by SP method and SP-MM method with $m=0,1$ are identical to each other, and it is because moments of $\partial_{\textbf{n}\textbf{m}} \rho$  vanish up to order $1$ automatically.

	To demonstrate the numerical performance of the proposed methods, we compare them with the well-established KTM and ATKM methods in terms of implementation, accuracy and efficiency. Among these methods, our methods are the easiest to implement.
Figure \ref{Fig:Compare} presents log-log plots of computational time versus errors for the 2D Coulomb potential with density
$\rho = e^{-4(x^2+y^2)}$. From these results, we observe that when the accuracy requirement is not very high (e.g., $ \ge 10^{-11}$), our methods are more efficient.

\begin{figure}[!htpb]
\centering
\includegraphics[scale=0.44]{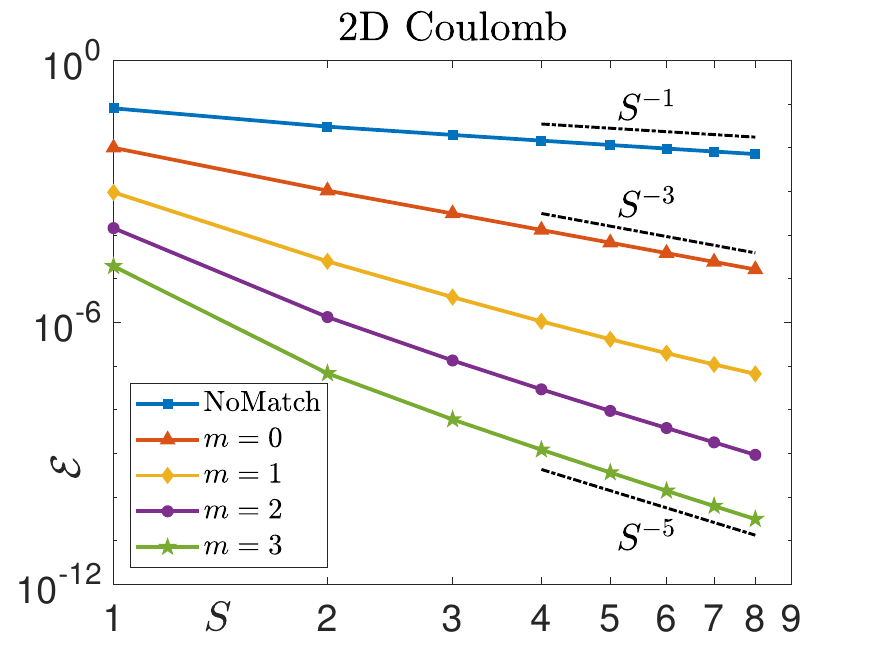}
\hspace{-1em}
\includegraphics[scale=0.44]{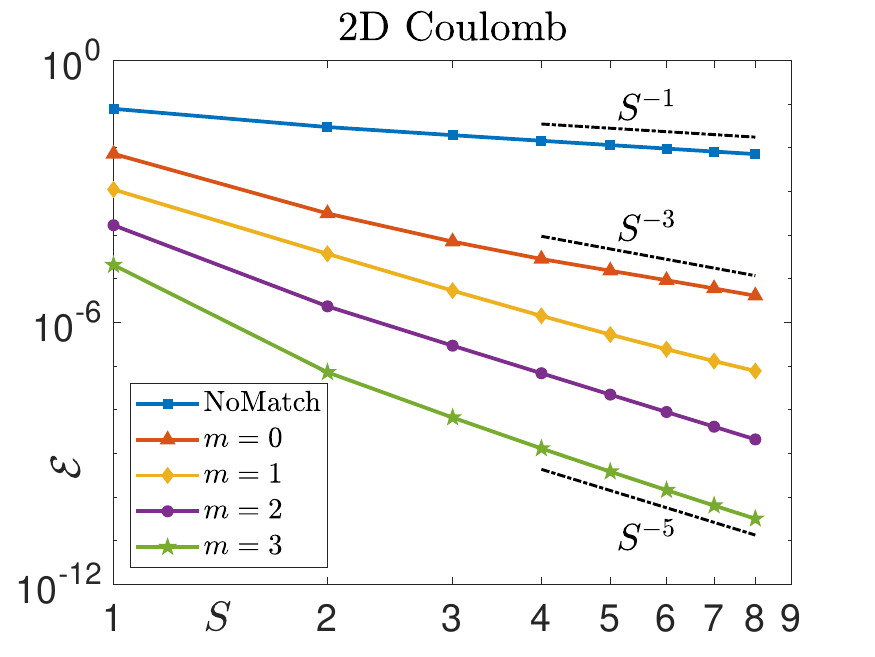} \\[0.5em]
\includegraphics[scale=0.44]{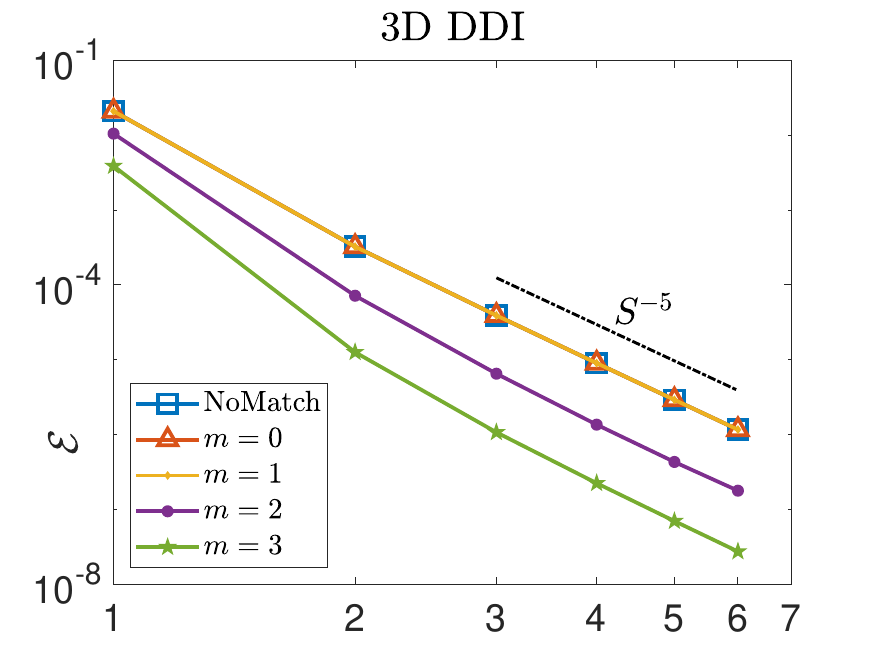} \hspace{-1em}
\includegraphics[scale=0.44]{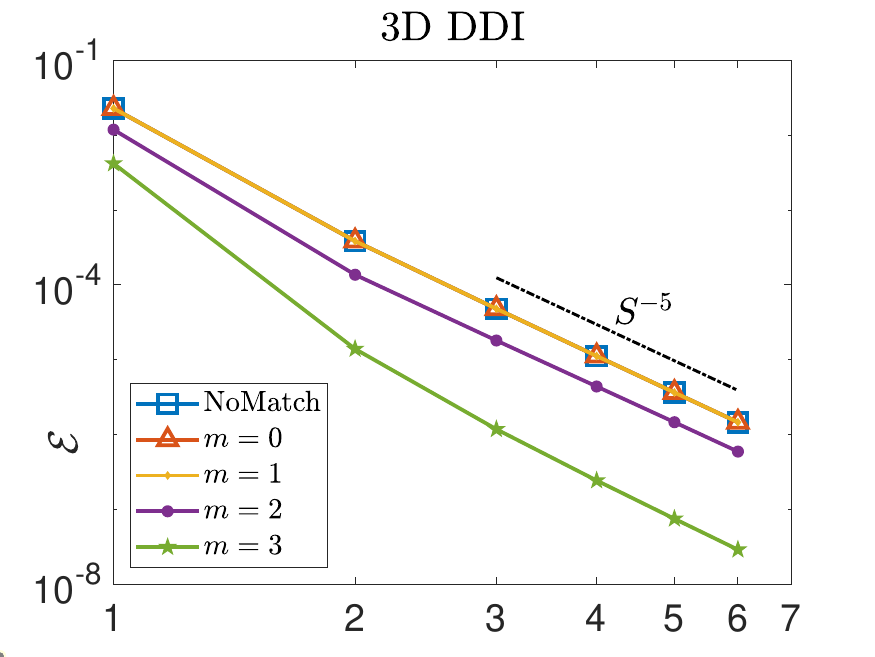} \\
\caption{Errors of $\mathrm{SP\text{-}MM}$ (left) and  $\mathrm{FS\text{-}MM}$ (right) methods for Coulomb potentials and DDI. }
\label{Fig:CoulombDDIAccu}
\end{figure}

\begin{figure}[!htpb]
	\centering
	\includegraphics[scale=0.5]{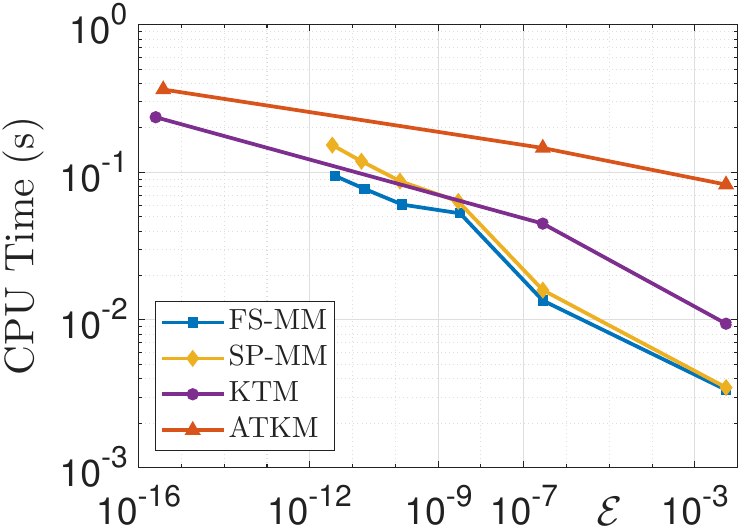}
	\caption{Log-log plots of timing results versus errors for Coulomb potential in Example \ref{Exmp:CoulombAccu}.  }
	\label{Fig:Compare}
\end{figure}

\subsection{The Biharmonic potentials in 2D/3D}

\begin{exmp}\label{exmp:Biharmonic}
Here, we consider the 2D and 3D Biharmonic potentials with convolution kernel
\begin{equation*}
      U(\mathbf{x}) =
      \left\{\begin{array}{ll}
-\frac{1}{8\pi} |\mathbf{x}|^2 \left(\ln(|\mathbf{x}|)-1\right),  &d=2,  \\[0.5em]
\frac{1}{8\pi} |\mathbf{x}|, & d=3. \\
\end{array}\right.
   \end{equation*}
\end{exmp}

Figure \ref{Fig:BiharmonicAccu} shows the errors and convergence orders with respect to the domain expansion factor $S$. From these results, we can see that the convergence order is at least $ m-1$ for the 2D Biharmonic potential and $m$ for the 3D Biharmonic potential, which is consistent with the theoretical results.

\begin{figure}[!htpb]
\centering
\includegraphics[scale=0.44]{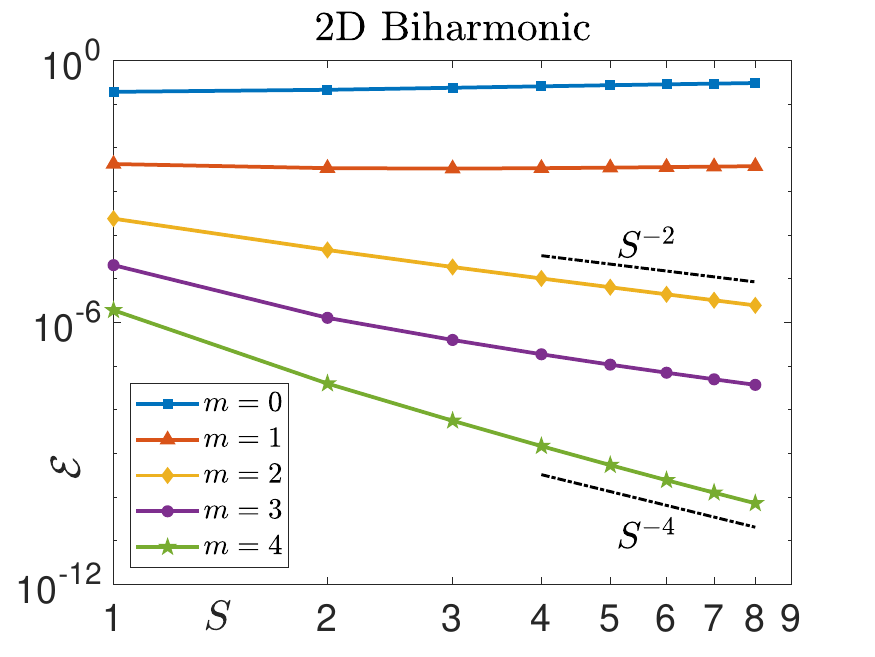}
\hspace{-1em}
\includegraphics[scale=0.44]{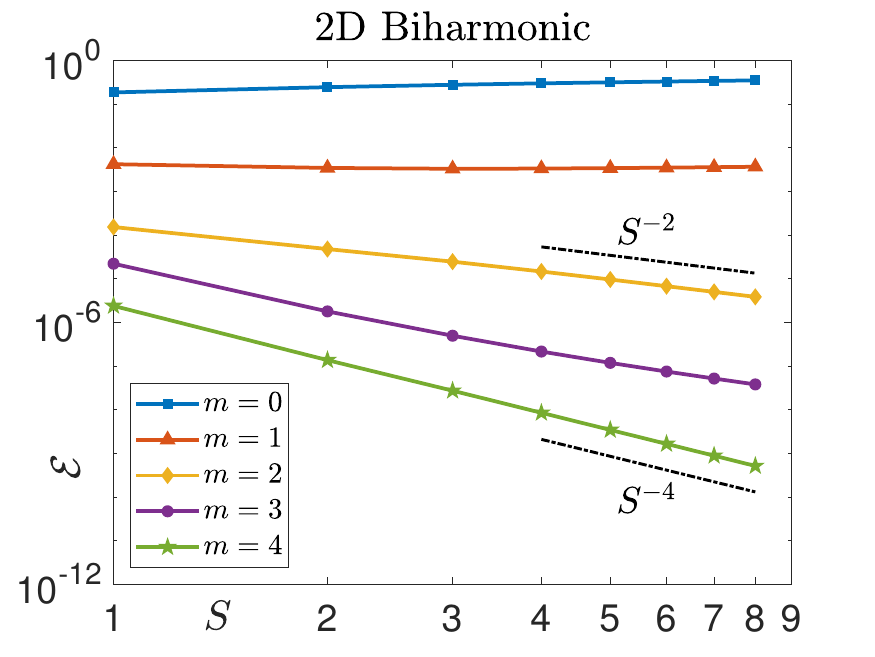} \\[0.5em]
\includegraphics[scale=0.44]{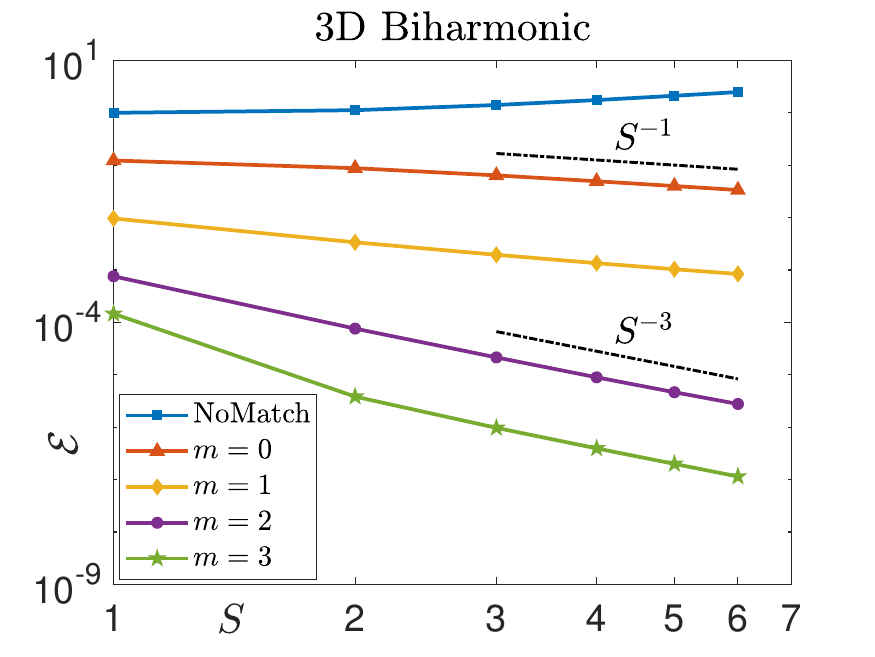} \hspace{-1em}
\includegraphics[scale=0.44]{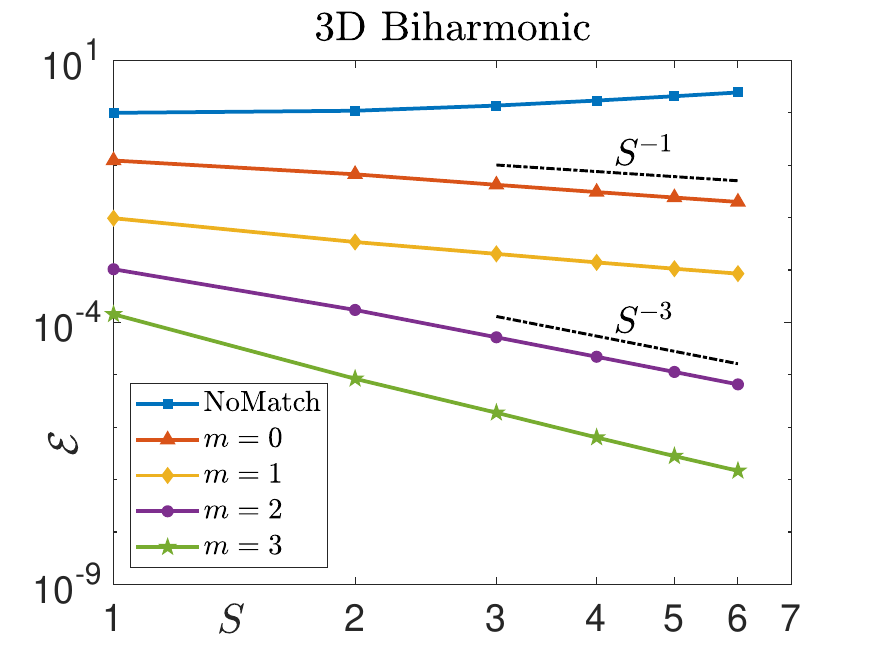} \\
\caption{Errors of the $\mathrm{SP\text{-}MM}$ (left) and   $\mathrm{SP\text{-}Tensor}$ (right) methods for Biharmonic potentials. }
\label{Fig:BiharmonicAccu}
\end{figure}

\subsection{The Yukawa potentials in 2D/3D}

\begin{exmp}\label{exmp:Yukawa}
Here, we consider the 2D and 3D Yukawa potentials with convolution  kernel
\begin{equation*}
      U(\mathbf{x}) =
      \left\{\begin{array}{ll}
\frac{1}{2 \pi} \operatorname{K_0} (\lambda |\mathbf{x}|),  &d=2,  \\[0.5em]
\frac{1}{4 \pi |\mathbf{x}|}e^{-\lambda |\mathbf{x}|}, & d=3, \\
\end{array}\right.
   \end{equation*}
   where $\operatorname{K}_0(r)$ is modified Bessel functions of the second kind with order 0 \cite{Handbook}.
For a Gaussian density $\rho( \bx ) =e^{-|\bx|^2/ \sigma ^2}$, the Yukawa potentials are given explicitly as \cite{atkm}
\beas
\varphi(\bx) =
  \left\{\begin{array}{ll}
 \int_0^{\infty} \operatorname{K}_0(\lambda s ) s ~e^{-\frac{r^2+s^2}{\sigma^2}}~ \operatorname{I}_0 \Big(\frac{2 r s }{\sigma^2}\Big)~ {\rm d} s
  , & d = 2,\\[0.5em]
 (\sqrt{\pi} \sigma)^3 ~ \frac{e^{-\lambda r + \frac{\lambda^2 \sigma^2}{4}}}{8 \pi r}
\left[ \operatorname{erfc}
\left( -\frac{r}{ \sigma} + \frac{\lambda \sigma}{2}
\right) -e^{2 \lambda r} \operatorname{erfc} \left( \frac{r}{\sigma} + \frac{\lambda \sigma}{2} \right)
\right],
      & d= 3,
         \end{array}\right.
\eeas
where  $\operatorname{I}_0(r)$ is the modified Bessel function of order zero and
 $\operatorname{erfc}(r) = 1- \frac{2}{\sqrt{\pi}} \int_{0}^{r} e^{-t^2} \rm{d}t $ is the complementary error function \cite{Handbook}.

\end{exmp}

Note that the kernel $\operatorname{K}_0(\lambda r)$ and $e^{-\lambda r}/r$ decay exponentially to zero at infinity and the decay rate grows as $|\lambda|$ gets larger. This property makes the SP method highly effective.
Table \ref{tab:YukawaAccu} presents the errors with different $\lambda$ and $S$.
 The results demonstrate that, with the help of the expanding domain technique, this method can readily achieve high accuracy.

\begin{table}
\centering
\caption{Errors of  2D/3D Yukawa potentials in Example \ref{exmp:Yukawa}.}
\label{tab:YukawaAccu}
\begin{tabular}{ccccc}
\toprule
 & & $ S = 1 $ & $S = 2$ & $S=3 $  \\
  \midrule
\multirow{3}{*}{2D}
 & $\lambda = 1$ & 1.9058E-07 & 3.7234E-16 & 3.7234E-16  \\
  & $\lambda = 2$    & 5.3590E-13 & 3.5729E-16 & 4.0353E-16  \\
    & $\lambda = 3$    & 4.2084E-16 & 3.6387E-16 & 4.1278E-16  \\
 \midrule
\multirow{3}{*}{3D}
 & $\lambda = 1$  & 8.2316E-08&  4.0121E-16 & 6.8779E-16   \\
 & $\lambda = 2$  & 2.8871E-13&  3.5191E-15 &  3.5191E-15  \\
   & $\lambda = 3$    & 6.3361E-15 & 6.3361E-15 & 6.6241E-15  \\
\bottomrule
\end{tabular}
\end{table}

\section{Conclusions} \label{Conclusion}
We proposed two simple and accurate fast algorithms to compute the convolution potential based on moment-matching for the density.
Each method requires merely minor modifications to popular sine pseudospectral (SP)/Fourier spectral (FS) method
and achieves arbitrary high order convergence.
To this end, we split the potential into two part, i.e., $\varphi = U \ast \rho_1 +  U \ast (\rho-\rho_1):= \varphi_1+\varphi_2$,
where the moments of auxiliary function $\rho_1$ match those of $\rho$ up to order $m$.
The auxiliary function $\rho_1$ is constructed as a linear combination of Gaussian and its derivatives,
and potential $\varphi_1$ can be computed analytically or integrated numerically with ease.
For the smooth residual density $\rho_2 = \rho-\rho_1$, its moments vanish up to order $m$
and the Fourier transform $\widehat{\rho_2}(\bk)$ decays exponentially fast at the far field.

In the first method, one solves a differential/pseudo-differential equation with homogeneous Dirichlet boundary conditions using SP.
While, in the second method, we discretize the Fourier integral by trapezoidal/midpoint rule.
Rigorous error estimates were provided to confirm the $\mathcal{O}(L^{-p})$ convergence on domain $\Omega =[-L,L]^d$ for both methods,
where the integer $p$ depends on order $m$ and the kernel $U$. To further improve the numerical accuracy and efficiency,
we employ the domain expansion technique and simplify each resulting quadrature into one discrete convolution.
The execution finally boils down to discrete Fast Fourier transform on a double-sized vector,
which is of essential importance for long-time simulation in terms of efficiency,
especially when the potential is evaluated multiple times under the same setups.


\section*{Acknowledgements}
This
work was partially supported by the National Key R\&D Program of  China (No. 2024YFA1012803), the National Natural Science Foundation of China (No. 11971335) and the Institutional Research Fund from Sichuan University (No. 2020SCUNL110) (Q. Tang), and the National Natural Science Foundation of China (No.12271400) and Tianyuan Mathematical center in Southwest China (No. 12226102) (X. Liu and
Y. Zhang).
\appendix
\section{Coefficients $\texorpdfstring{\gamma_{\bm{\alpha}}}{\gamma_{\alpha}}$}
\label{coefficient}
\renewcommand{\theequation}{A.\arabic{equation}}
\setcounter{equation}{0}
\setcounter{table}{0}
\renewcommand{\thetable}{A\arabic{table}}

 Specifically, the coefficients $\gamma_{\balpha}$ are determined by the linear system \eqref{LineEqu}.
For the sake of simplicity, we define $H^{\balpha}_{\bbeta} : =  \int_{\mathbb{R}^d} \frac{\partial^{\balpha} G(\bx)}{\partial \bx^{\balpha}} \bx^{\bbeta} \rm{d} \bx$ and $P_{\bbeta}:= \int_{\mathbb{R}^d} \rho(\bx) \bx^{\bbeta} \rm{d} \bx$, then, the linear system \eqref{LineEqu} can be rewritten as
\bea
\label{SimplifyCoef}
\sum_{|\balpha|=0}^{m} H^{\balpha}_{\bbeta} ~ \gamma_{\balpha}=P_{\bbeta},  ~~~~~ |\bbeta|=0,\dots,m.
\eea
 Here, $H^{\bm{\alpha}}_{\bm{\beta}}$ can be evaluated analytically
	and $P_{\bm{\beta}}$ can be well approximated using the trapezoidal quadrature.
Given a fixed $m$, we need to compute $\frac{(m+d)!}{m!d!}$ coefficients for $d=1, 2, 3$.
Here, we claim that matrix $H^{\balpha}_{\bbeta}$ is lower triangular if the auxiliary  function $G(\bx)$  is symmetric and compactly supported.
This observation suggests the solution to \eqref{SimplifyCoef} can be obtained by back substitution.
 Table \ref{AppCoeff2D} and Table \ref{AppCoeff3D} list the explicit formulas for coefficients $\gamma_{\balpha}$.
 Taking the 2D case as an example, the matrix $H^{\balpha}_{\bbeta}$ has  the following  properties.
\vspace{0.5em}
\begin{prop} \label{PropMatrix}
Denote the coefficient  matrix by
 \beas
 H^{(\alpha_1, \alpha_2)}_{(\beta_1, \beta_2)}=\int_{\mathbb{R}^2} \frac{\partial^{(\alpha_1 + \alpha_2)} G(\bx)}{\partial x_1 ^{\alpha_1} x_2 ^{\alpha_2}} x_1^{\beta_1} x_2^{\beta_2}  {\rm{d}} x_1  {\rm{d}} x_2.
  \eeas
  If the function  $G(\bx)$ is symmetric and compactly supported, the following properties
  \vspace{0.2em}
\begin{enumerate}[label=(\roman*)]
  \item If $\alpha_i + \beta_i $ is odd for some $i$, then $H^{(\alpha_1, \alpha_2)}_{(\beta_1, \beta_2)} = 0$,
  \item Symmetry:  $H^{(\alpha_1, \alpha_2)}_{(\beta_1, \beta_2)} = H^{(\alpha_2, \alpha_1)}_{(\beta_2, \beta_1)}$,
  \item  Recurrence formula:  $H^{(\alpha_1, \alpha_2)}_{(\beta_1, \beta_2)} = - \beta_1 H^{(\alpha_1-1, \alpha_2)}_{(\beta_1-1, \beta_2)}$,
  \item If $\alpha_i > \beta_i $ for some $i$, then  $H^{(\alpha_1, \alpha_2)}_{(\beta_1, \beta_2)} = 0$,
\end{enumerate}
\vspace{0.2em}
hold true.
\end{prop}

\vspace{0.5em}

\begin{proof}
It is straightforward to verify that properties $(i) $ and $(ii)$ hold.
For property $(iii)$, using integration by parts,
we have
\beas
H^{(\alpha_1, \alpha_2)}_{(\beta_1, \beta_2)} &=& \int_{-\infty}^{\infty} \left[  \int_{-\infty}^{\infty} \frac{\partial^{(\alpha_1 + \alpha_2)} G(\bx)}{\partial x_1 ^{\alpha_1} x_2 ^{\alpha_2}} x_1^{\beta_1}   {\rm{d}} x_1 \right] x_2^{\beta_2} ~ {\rm{d}} x_2 \\[0.5em]
  & = & \int_{-\infty}^{\infty}
\left[  - \beta_1 \int_{-\infty}^{\infty} \frac{\partial^{(\alpha_1 + \alpha_2-1)} G(\bx)}{\partial x_1 ^{\alpha_1-1} x_2 ^{\alpha_2}} x_1^{\beta_1-1} ~  {\rm{d}} x_1 \right] x_2^{\beta_2} ~ {\rm{d}} x_2 = - \beta_1 H^{(\alpha_1-1, \alpha_2)}_{(\beta_1-1, \beta_2)}.
\eeas
For property $(iv)$, due to the fact that  function $G(\bx)$ is compactly supported, we have the following assertion:
\begin{equation} \label{assertion} \text{If}~ \alpha_i \not = 0, \beta_i =0 ~~  \text{for some} ~ i, \text{then}~  H^{(\alpha_1, \alpha_2)}_{(\beta_1, \beta_2)} = 0.
\end{equation}
Without loss of generality, we assume that  $\alpha_1 \not = 0$, $\beta_1 = 0$, then
\beas
H^{(\alpha_1, \alpha_2)}_{(\beta_1, \beta_2)} = \int_{-\infty}^{\infty} \left[  \int_{-\infty}^{\infty} \frac{\partial^{(\alpha_1 + \alpha_2)} G(\bx)}{\partial x_1 ^{\alpha_1} x_2 ^{\alpha_2}}    {\rm{d}} x_1 \right] x_2^{\beta_2}~  {\rm{d}} x_2 = 0.
\eeas
Combining property $(iii)$ and \eqref{assertion}, we see that if $\alpha_1 > \beta_1 $,
\beas
H^{(\alpha_1, \alpha_2)}_{(\beta_1, \beta_2)} = (-1)^{\beta_1} \beta_1 !  ~  H^{(\alpha_1-\beta_1, \alpha_2)}_{(0, \beta_2)} =0.
\eeas
Obviously, using the  symmetry property $(ii)$, we have  $H^{(\alpha_1, \alpha_2)}_{(\beta_1, \beta_2)} = 0$ if  $\alpha_2 > \beta_2 $.
\end{proof}

\begin{thm} \label{LowTriangle}
The matrix $H^{\balpha}_{\bbeta} = \int_{\mathbb{R}^2} \frac{\partial^{(\alpha_1 + \alpha_2)} G(\bx)}{\partial x_1 ^{\alpha_1} x_2 ^{\alpha_2}} x_1^{\beta_1} x_2^{\beta_2}  {\rm{d}} x_1  {\rm{d}} x_2 $
is lower triangular if function  $G(\bx)$  is symmetric and compactly supported.
\end{thm}

\vspace{0.5em}

\begin{proof}
It is sufficient to show that $H^{(\alpha_1, \alpha_2)}_{(\beta_1, \beta_2)} = 0$ if  $\alpha_1 + \alpha_2 \ge \beta_1 + \beta_2$ and $(\alpha_1, \beta_1)\not = (\alpha_2, \beta_2)$, which can be proved by  $(iv)$ in Proposition \ref{PropMatrix}.
 To be specific,
if $\alpha_1 > \beta_1$, it follows from  $(iv)$  that  $H^{(\alpha_1, \alpha_2)}_{(\beta_1, \beta_2)} = 0$.  Otherwise,  if $\alpha_1 \le \beta_1$, which means that $\alpha_2 > \beta_2$, then  $H^{(\alpha_1, \alpha_2)}_{(\beta_1, \beta_2)} = 0$ by  $(iv)$ .
\end{proof}

\begin{table}[!htbp]
\label{AppCoeff2D}
\renewcommand\arraystretch{1.4}
\centering
\caption{Coefficients $\gamma_{\balpha}$ with $|\balpha| = 0, 1, 2, 3, 4$ for the 2D case.}
\resizebox{.99\columnwidth}{!}{
\begin{tabular}{lllll}
\toprule
$ \gamma_{00}=P_{00}$  &  &  &  &    \\
\midrule
$\gamma_{01} = -P_{01}$  & $\gamma_{10} = -P_{10}$  &  &  &    \\
 \midrule
$\gamma_{02} = \frac{1}{2} (P_{02} - \sigma^2 \gamma_{00})$  & $\gamma_{11} = P_{11}$  & $\gamma_{20} = \frac{1}{2} (P_{20} - \sigma^2 \gamma_{00})$  &  &    \\
  \midrule
$\gamma_{03} = \frac{1}{6}(-P_{03}- 3 \sigma^2 \gamma_{01})$   & $\gamma_{12} = \frac{1}{2}(-P_{12} - \sigma^2 \gamma_{10})$  & $\gamma_{21} = \frac{1}{2}(-P_{21} - \sigma^2 \gamma_{01})$  & $\gamma_{30} = \frac{1}{6}(-P_{30}- 3 \sigma^2 \gamma_{10})$ &    \\
 \midrule
$\gamma_{04} = \frac{1}{24} (P_{04} - 3 \sigma^4 \gamma_{00} -12 \sigma^2 \gamma_{02})$  & $\gamma_{13} = \frac{1}{6} (P_{13}-3 \sigma^2 \gamma_{11})$  & $\gamma_{22} =\frac{1}{4}(P_{22}-\sigma^4 \gamma_{00}-2 \sigma^2 \gamma_{20} - 2\sigma^2 \gamma_{02})$  & $\gamma_{31} = \frac{1}{6} (P_{31}-3 \sigma^2 \gamma_{11})$ &  \\
 $\gamma_{40} = \frac{1}{24} (P_{40} - 3 \sigma^4 \gamma_{00} -12 \sigma^2 \gamma_{20})$ &&&&\\
\bottomrule
\end{tabular}}
\end{table}

\begin{table}[!htbp]
\label{AppCoeff3D}
\renewcommand\arraystretch{1.4}
\centering
\caption{Coefficients $\gamma_{\balpha}$ with $|\balpha| = 0, 1, 2, 3$ for the 3D case..}
\resizebox{.99\columnwidth}{!}{
\begin{tabular}{llll}
\toprule
$ \gamma_{000}=P_{000}$  &  &  &      \\
\midrule
$\gamma_{001} = -P_{001}$  & $\gamma_{010} = -P_{010}$  & $\gamma_{100} = -P_{100}$ &      \\
 \midrule
$\gamma_{002} = \frac{1}{2} (P_{002} - \sigma^2 \gamma_{000})$  & $\gamma_{011} = P_{011}$  & $\gamma_{020} = \frac{1}{2} (P_{020} - \sigma^2 \gamma_{000})$  & $\gamma_{101} =P_{101}$         \\
$\gamma_{110} = P_{110} $ & $\gamma_{200} =\frac{1}{2} (P_{200} - \sigma^2 \gamma_{000}) $ & & \\
  \midrule
$\gamma_{003} = \frac{1}{6}(-P_{003}- 3 \sigma^2 \gamma_{001})$   & $\gamma_{012} = \frac{1}{2}(-P_{012} - \sigma^2 \gamma_{010})$  & $\gamma_{021} = \frac{1}{2}(-P_{021} - \sigma^2 \gamma_{001})$  &   $\gamma_{030} = \frac{1}{6}(-P_{030}- 3 \sigma^2 \gamma_{010})$   \\
$\gamma_{102} = \frac{1}{2}(-P_{102} - \sigma^2 \gamma_{100}) $  &$\gamma_{111} = -P_{111} $& $\gamma_{120} = \frac{1}{2}(-P_{120} - \sigma^2 \gamma_{100}) $ & $\gamma_{201} = \frac{1}{2}(-P_{201} - \sigma^2 \gamma_{001})$ \\
  $\gamma_{210} = \frac{1}{2}(-P_{210} - \sigma^2 \gamma_{010})$& $\gamma_{300} = \frac{1}{6}(-P_{300}- 3 \sigma^2 \gamma_{100})$ && \\
\bottomrule
\end{tabular}}
\end{table}

\section{Analytical expression  $W(\bx)$} \label{anlyticalExpr}

\begin{itemize}
\item
\textbf{The d-dimensional Poisson potential.}
 \beas
 W(\bx)=\left\{\ba{ll}
-\frac{1}{4 \pi} \left[\operatorname{E}_1\big(\frac{r^2}{2 \sigma^2}\big) + 2 \ln(r) \right], ~~ \qquad  d=2, \\
\\
 \frac{1}{4 \pi |\bx|} \operatorname{Erf}\left(\frac{r}{\sqrt{2} \sigma}\right),~~~~~~~~~~~~~~ \qquad d=3.
\ea\right.
\eeas
Here, $\operatorname{E}_1(r)=\int_r^{\infty} t^{-1} e^{-t} \rm{d}t$ is the exponential integral function  and  $\operatorname{Erf}(r)= \frac{2}{\sqrt{\pi}} \int_0^r e^{-t^2} \rm{d}t $ is the error function \cite{Handbook, atkm}.
\item
\textbf{The 2D Coulomb potential.}
\beas
 W(\bx)=\frac{1}{2\sqrt{2 \pi} \sigma} \operatorname{I}_0 \left(\frac{r^2}{4 \sigma^2}\right)e^{-\frac{r^2}{4 \sigma^2}},
\eeas
where  $\operatorname{I}_0(r)$ is the modified Bessel function of order zero \cite{Handbook,atkm}.
\item
\textbf{The d-dimensional Biharmonic potential.}
 \beas
 W(\bx)=\left\{\ba{ll}
 \frac{1}{8 \pi} \left[ r^2+ e^{-\frac{r^2}{\sigma^2}} \sigma^2 \right]+
\frac{1}{16} (r^2+2 \sigma^2) \left[ \operatorname{Ei}( -\frac{r^2}{2 \sigma^2}  ) -2 \ln(r) \right], ~~ \quad d=2, \\
\\
 \frac{1}{8 \pi} \left[ \operatorname{Erf} \left(\frac{r}{\sqrt{2} \sigma} \right)(\frac{\sigma^2}{r}+r)+
\sigma \sqrt{\frac{2}{\pi}} e^{-\frac{r^2}{2 \sigma^2}}
 \right],~~~~~~~~\qquad \qquad \qquad d=3.
\ea\right.
\eeas
Here,  $\operatorname{Ei}(r):= \int_{-\infty}^{r} t^{-1}e^t \rm{d}t$ is the exponential integral \cite{Handbook, atkm}.
\end{itemize}

\end{document}

%% file: command.tex
\newcommand{\be}{\begin{equation}}
\newcommand{\ee}{\end{equation}}
\newcommand{\ba}{\begin{array}}
\newcommand{\ea}{\end{array}}
\newcommand{\bea}{\begin{eqnarray}}
\newcommand{\eea}{\end{eqnarray}}
\newcommand{\beas}{\begin{eqnarray*}}
\newcommand{\eeas}{\end{eqnarray*}}

\newcommand{\bx}{{\mathbf x}}
\newcommand{\by}{{\mathbf y}}

\newcommand{\bp}{{\mathbf p}}
\newcommand{\im}{i}

\newcommand{\balpha}{{\bm \alpha}}
\newcommand{\bbeta}{{\bm \beta}}

\newcommand{\bh}{{\textbf h}}

\newcommand{\bk}{{\textbf k}}
\newcommand\bR{\mathbf{R}}

\renewcommand{\theequation}{\arabic{section}.\arabic{equation}} %
\newtheorem{exmp}{Example}
\newtheorem{remark}{Remark}[section]
\newtheorem{prop}{Proposition}
\newtheorem{thm}{Theorem}
\newtheorem{lem}{Lemma}

